\documentclass[12pt]{amsart}
\usepackage[mathscr]{eucal}
\usepackage{amsfonts}
\usepackage{amsmath}
\usepackage{amsthm}
\usepackage{amssymb}
\usepackage{latexsym}

\usepackage{ytableau}

\usepackage{graphicx}
\usepackage{color}
\usepackage[noadjust]{cite}
\usepackage{epsfig}

\newfont{\msam}{msam10}
\newtheorem{theorem}{Theorem}[section]
\newtheorem{theorem_noname}{Theorem}
\newtheorem{proposition}[theorem]{Proposition}
\newtheorem{corollary}[theorem]{Corollary}
\newtheorem{lemma}[theorem]{Lemma}
\theoremstyle{definition}
\newtheorem{definition}[theorem]{Definition}
\newtheorem{remark}[theorem]{Remark}

\let\nc\newcommand

\nc{\la}{\label}

\def\bthm{\begin{theorem}}
\def\ethm{\end{theorem}}
\def\blemma{\begin{lemma}}
\def\elemma{\end{lemma}}
\def\bproof{\begin{proof}}
\def\eproof{\end{proof}}
\def\bprop{\begin{proposition}}
\def\eprop{\end{proposition}}

\def\Z{\mathbb{Z}}

\def\O{\mathcal{O}}

\def\N{\mathbb{N}}

\def\RR{R}

\def\H{\mathrm{\ddot H}}

\def\e{\boldsymbol{\mathrm{e}}}

\def\q{\mathbf{q}}
\def\g{\mathfrak{g}}

\def\sl{\mathfrak{sl}}

\def\C{\mathbb{C}}

\nc{\Hom}{{\rm{Hom}}}
\nc{\Ext}{{\rm{Ext}}}
\nc{\htau}{{\bar{ t}}}
\nc{\HOM}{\underline{\rm{Hom}}}
\nc{\EXT}{\underline{\rm{Ext}}}
\nc{\TOR}{\underline{\rm{Tor}}}
\nc{\End}{{\rm{End}}}
\nc{\Map}{{\rm{Map}}}
\nc{\Out}{{\rm{Out}}}
\nc{\GL}{{\rm{GL}}}
\nc{\SL}{{\rm{SL}}}
\nc{\PGL}{{\rm{PGL}}}
\nc{\G}{{\rm{G}}}
\nc{\Rep}{{\rm{Rep}}}
\nc{\ad}{{\rm{ad}}}
\nc{\dlim}{\varinjlim}

\newcommand{\ev}{{\rm{ev}}}

\numberwithin{equation}{section}

\newcommand{\chr}{{\mathrm{Char}}}

\newcommand{\xx}{{\mathbf{x}}}
\newcommand{\yy}{{\mathbf{y}}}
\newcommand{\aab}{{\mathbf{a}}}
\newcommand{\bb}{{\mathbf{b}}}

\newcommand{\mm}{{\mathbf{m}}}
\newcommand{\nn}{{\mathbf{n}}}
\newcommand{\pmone}{{\pm 1}}

\def\SH{\mathrm {S} \H}

\def\E{\mathcal{E}}

\def\cc{\mathcal{C}}
\def\gl{\mathfrak{gl}}

\usepackage{pinlabel}
\newcommand{\pic}[2]{\raisebox{-.5\height}{\includegraphics[scale=#2]{#1}}}
\def\aijannulus{\pic{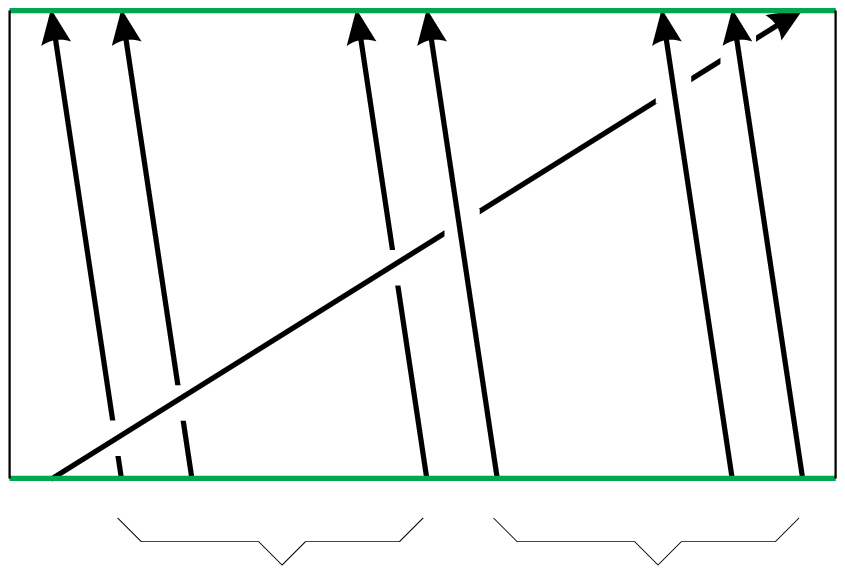}{.50}}
\def\aijtorus{\pic{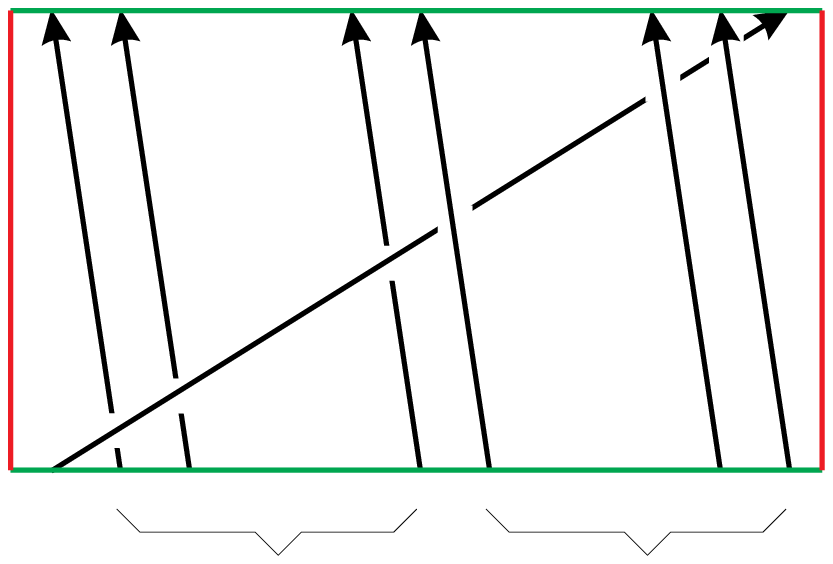}{.50}}
\def\Aaa{\pic{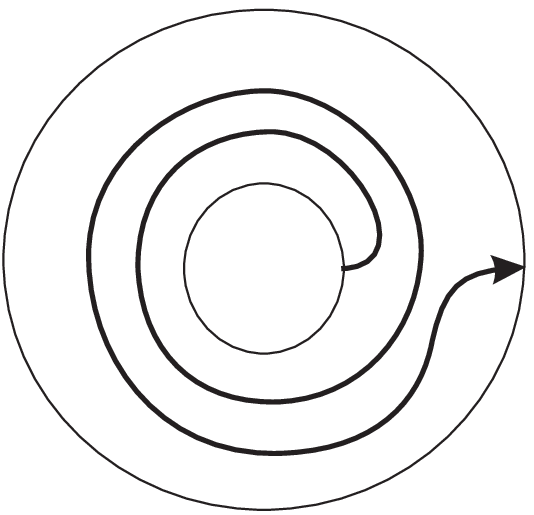} {.600}}
\def\Aeh{\pic{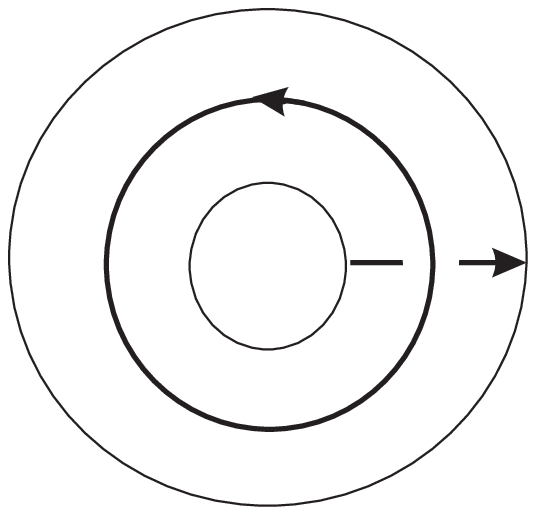} {.600}}
\def\Ahe{\pic{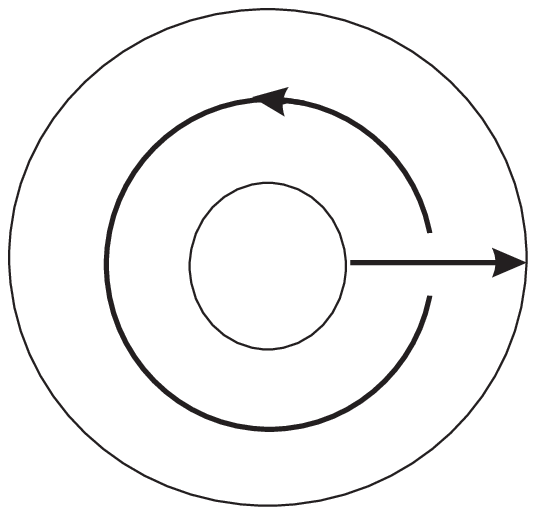} {.60}}
\def\Cij{\pic{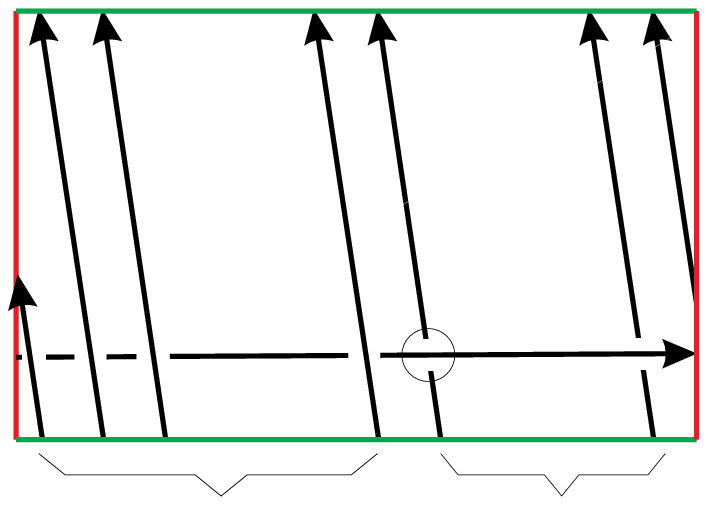}{.5}}
\def\smoothedtorus{\pic{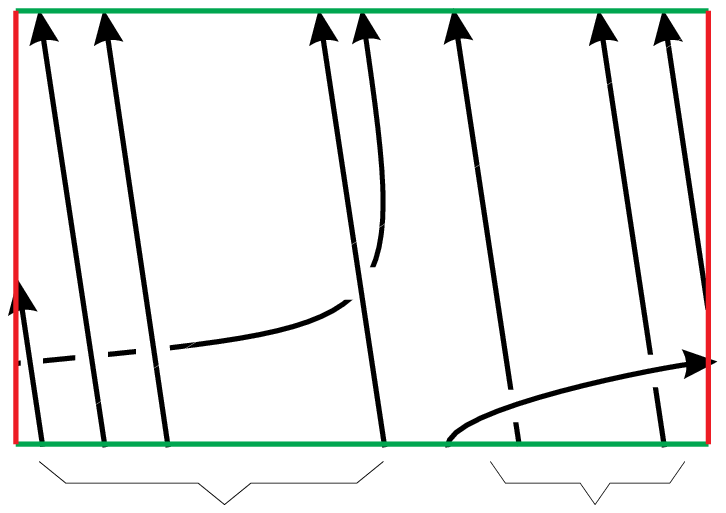}{.5}}
\def\back{\pic{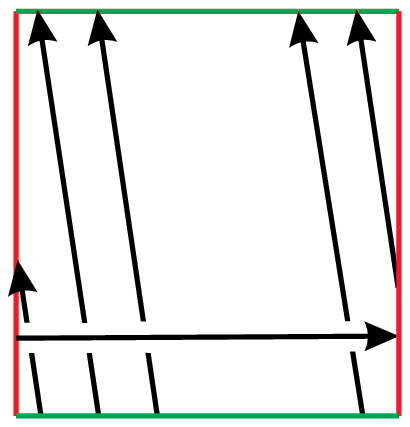}{.5}}
\def\front{\pic{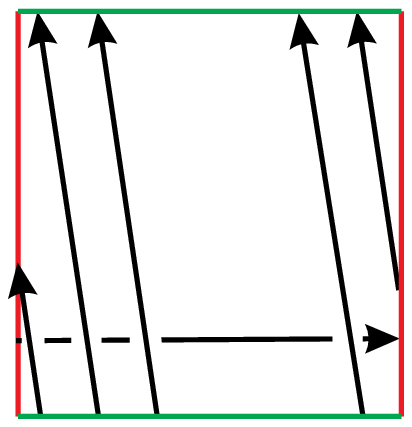}{.5}}
\def\minusonek{\pic{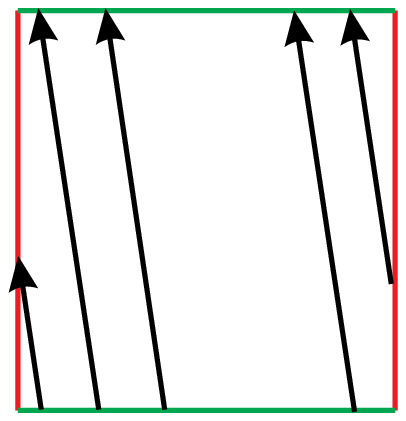}{.5}}

\newcommand\Xor{\pic{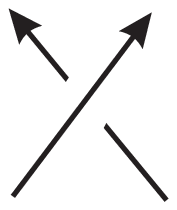}{.50}}
\newcommand\Yor{\pic{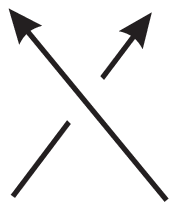} {.50}}
\newcommand\Ior{\pic{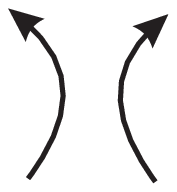} {.50}}

 \newcommand\Idor{\pic{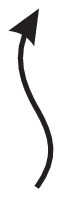} {.50}}
\newcommand\Rcurlor{\pic{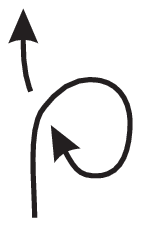} {.50}}
\newcommand\Lcurlor{\pic{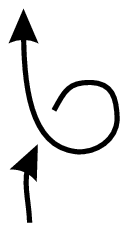} {.50}}
\newcommand\sigmaior{\pic{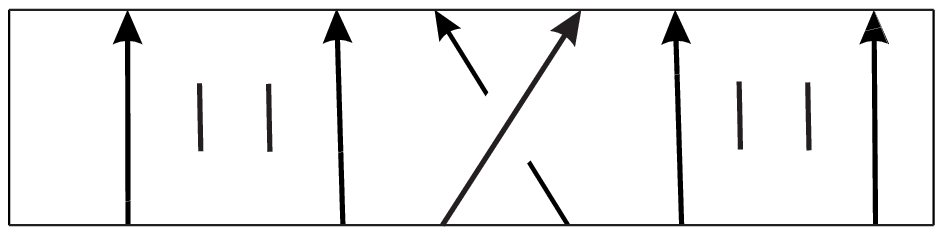} {.50}}

 \newcommand\Ponezero{\pic{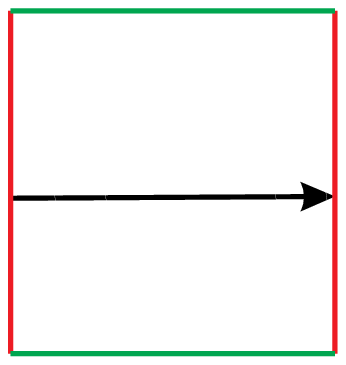}{.5}}
  \newcommand\Pzeroone{\pic{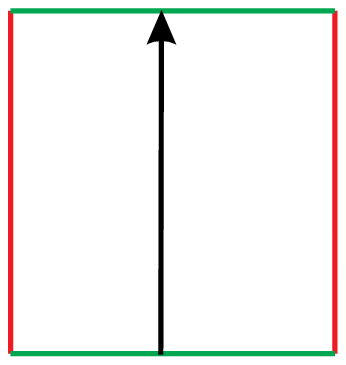}{.5}} 
  \newcommand\Ppos{\pic{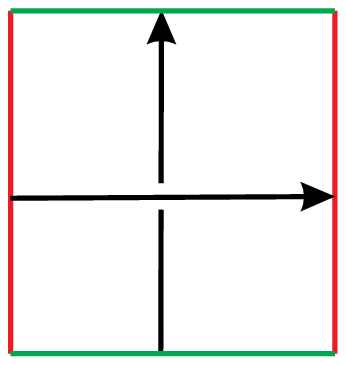}{.5}}
  \newcommand\Pneg{\pic{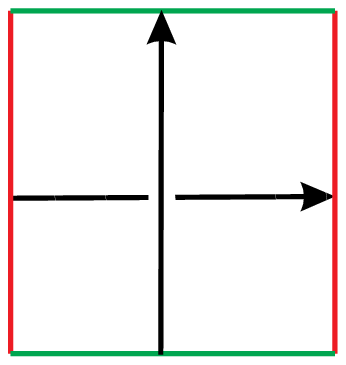}{.5}}
  \newcommand\Poneone{\pic{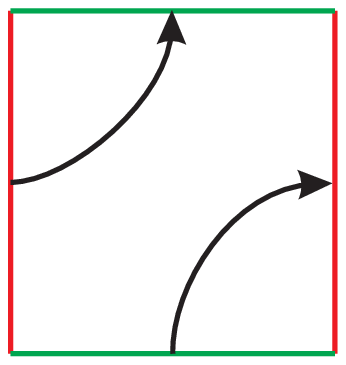}{.5}}
 
 \newcommand\Torus{\pic{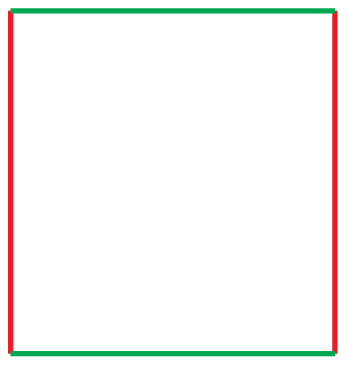}{.5}}
 \newcommand\backid{\pic{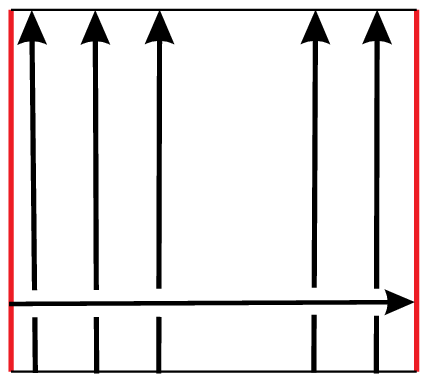}{.5}}
 \newcommand\frontid{\pic{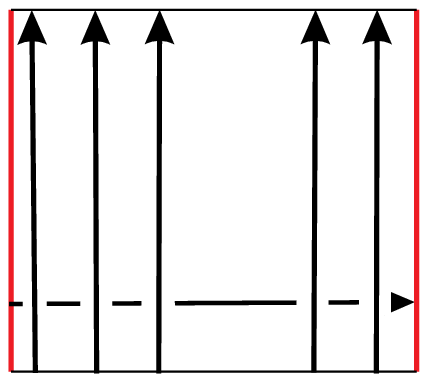}{.5}}
 \newcommand\Murphycylinder{\pic{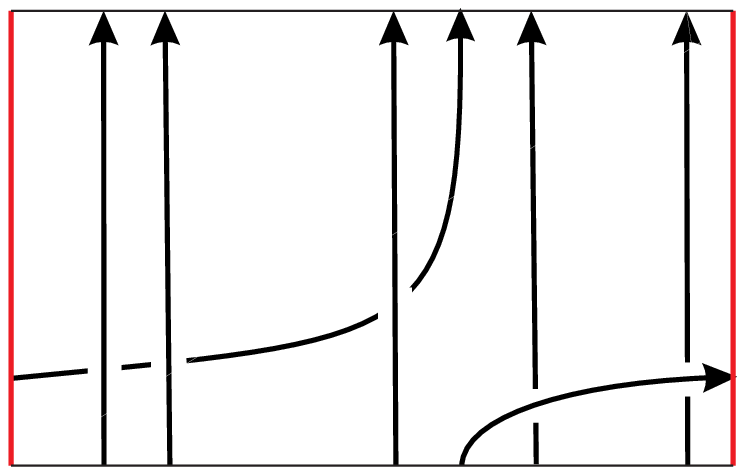}{.5}}
 \newcommand\CylinderAnnulus{\pic{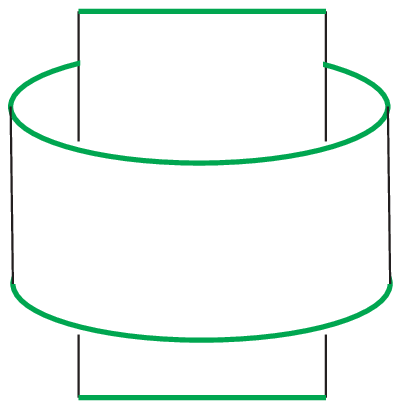}{.5}}
 \newcommand\torusD{\pic{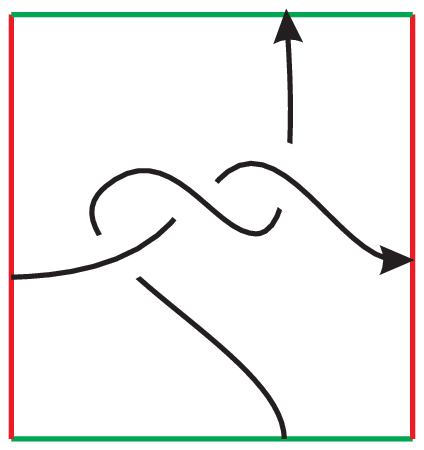}{.5}}
 \newcommand\torusE{\pic{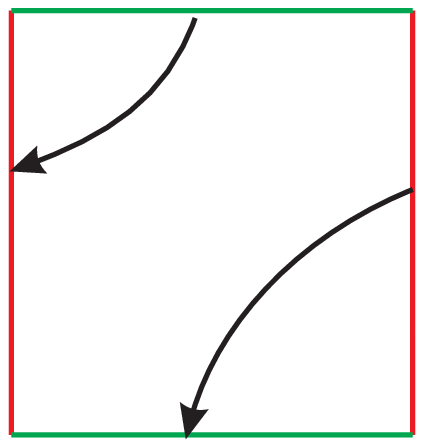}{.5}}
 \newcommand\torusDE{\pic{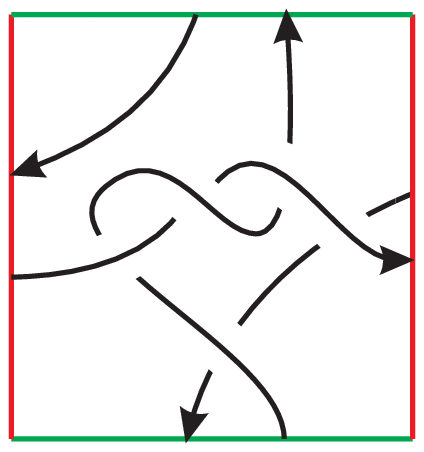}{.5}}
  \newcommand\torusED{\pic{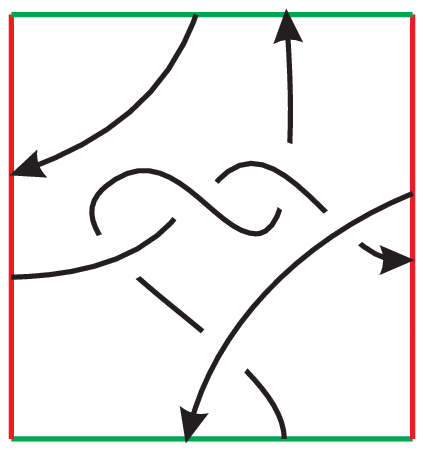}{.5}}

\newcommand\bc{\begin{center}}
\newcommand\ec{\end{center}}

\newcommand{\beqn}{\begin{eqnarray*}}
\newcommand{\eeqn}{\end{eqnarray*}}

\renewcommand{\r}{\textcolor{red}}

\begin{document}
\title[]{The Homflypt skein algebra of the torus and the elliptic Hall algebra}
\date{\today}

\author{Hugh Morton}
\address{Department of Mathematical Sciences, University of Liverpool, Peach Street, Liverpool L69 7ZL, UK}
\email{morton@liv.ac.uk}

\author{Peter Samuelson}
\address{Department of Mathematics, University of Toronto, Toronto, ON M4Y 1H5, Canada}
\email{psam@math.utoronto.ca}

\begin{abstract}
We give a generators and relations presentation of the Homflypt skein algebra $H$ of the torus $T^2$, and we give an explicit description of the module corresponding to the solid torus. Using this presentation, we show that $H$ is isomorphic to the $t=q$ specialization of the elliptic Hall algebra of Burban and Schiffmann \cite{BS12}.

As an application, for an iterated cable $K$ of the unknot, we use the elliptic Hall algebra to construct a 3-variable polynomial that specializes to the $\lambda$-colored Homflypt polynomial of $K$. We show that this polynomial also specializes to one constructed by Cherednik and Danilenko using the $\gl_N$ double affine Hecke algebra. This proves one of the Connection Conjectures in \cite{CD14}.

\end{abstract}
\maketitle
\setcounter{tocdepth}{1}
\tableofcontents

\section{Introduction}
In this paper we compare an algebra coming from knot theory with an algebra defined using elliptic curves over finite fields. We use this to derive results about polynomials associated to algebraic knots, which are iterated torus knots that arise from a singular point of a planar algebraic curve. Below we briefly introduce both algebras and then give a precise statements of the results.

\subsection{The Homflypt skein algebra}

The \emph{framed Homflypt skein module }$H(M)$ of an oriented $3$-manifold $M$ consists of $R$-linear combinations of framed oriented links in $M$ up to isotopy, modulo the linear `skein relations' 
\bc
$\Xor - \Yor=(s-s^{-1})\ \Ior$ \qquad (Switch and smooth)\\[2mm]

$\Rcurlor=v^{-1}\ \Idor\ ,\qquad \Lcurlor =v\ \Idor$ \qquad (Framing change)\\[2mm]
\ec
using the ring $R={\Z}[v^{\pm1},s^{\pm1}]$ with denominators $s^r-s^{-r},r>0$ as coefficient ring.

In these relations we use the convention that the framing is defined by a ribbon for each component of the link, and the skein relations apply in a ball in which the framing ribbons  lie parallel to the strands as seen. This is sometimes termed the `local blackboard framing'. The second framing relation then introduces $v^{\pm1}$ when a framing ribbon acquires a single extra twist.

\begin{remark}
The term \emph{skein module}  indicates that $H(M)$ is a module over $R$. Since many skeins are also modules over more complicated $R$-algebras, Morton and his co-workers often simply use the term \emph{skein}, and  suppress the qualifications \emph{framed Homflypt} also when the context is clear. 
\end{remark}

When $F$ is an orientable surface and $M=F\times I$ we adopt the notation $H(F)$ in place of $H(F\times I)$, and refer to $H(F)$ as the skein (module) of the surface $F$.  Framed links in $F\times I$ can be represented by diagrams in $F$, with the global `blackboard framing' from $F$, up to isotopy and Reidemeister moves $ R_{II}, R_{III}$. We can then regard elements of $H(F)$ as diagrams in $F$ modulo $ R_{II}, R_{III}$ and the skein relations. 

The skein $H(F)$ forms an algebra over $R$ under the product induced by placing one copy of $F\times I$ containing an element $D\in H(F)$ on top of another copy containing $E\in H(F)$ to  determine their product $DE$ (see Section \ref{sec_diagrams}).

In the case $\mathcal C=H(F)$ where $F$ is the annulus $S^1\times I$ this algebra is commutative and has been studied for some time. A recent account of some of its properties can be found in \cite{MM08}. It has an interpretation as the algebra of symmetric functions in a large number of commuting variables $x_1,\ldots,x_N$, and contains an element $P_m$ for each $m$ corresponding to the power sum $x_1^m+\cdots+x^m_N$. One representation of this element, due originally to Aiston \cite{Ais96}, is a multiple of the sum of $m$ explicit closed $m$-braids (see Section \ref{sec_annulus}).

The case of primary interest to us is the skein $H=H(T^2)$ for the surface $T^2$. As an algebra $H$ is  non-commutative, and can be generated by elements $P_{\xx}$, one for each ${\xx}\in {\Z}^2 \setminus \{0\}$, corresponding to free homotopy classes of curves in $T^2$. 

For a primitive ${\xx}=(m,n)\in {\Z}^2$ we represent $P_{\xx}$ by the oriented embedded $(m,n)$ curve on $T^2$. It is an immediate consequence of the switch and smooth skein relation that the commutator $[P_{(1,0)}, P_{ (0,1)}]$ satisfies $$[P_{(1,0)}, P_{ (0,1)}]=(s-s^{-1})P_{(1,1)}.$$
 The same switching and smoothing relation shows that $[P_{\bf x},P_{\bf y}] =(s-s^{-1})P_{{\bf x}+{\bf y}}$ when the primitive curves $\bf x$ and $\bf y$ cross once in the positive direction. Our main result is the following satisfyingly simple extension of these commutation relations (see Theorem \ref{thm_commutationrelations}).
 
\begin{theorem_noname}{\rm (Global switch and smooth)}. The commutator $[P_{\xx},P_{\yy}] $ in $H$ satisfies
\begin{equation}\label{eq_commintro} 
[P_{\bf x},P_{\bf y}] =(s^d-s^{-d})P_{{\bf x}+{\bf y}} 
\end{equation}
where $d=\det[{\bf x \ y}]$ is the signed crossing number of $\bf x$ with $\bf y$.\\
\end{theorem_noname}

In making the statement of Theorem \ref{thm_commutationrelations} we must also specify $P_{m\bf x}$ for any multiple of a primitive $\bf x$. This is defined by decorating the embedded curve $\bf x$ by the element $P_m$ from the skein $\mathcal C$ of the annulus. (See Definition \ref{def_px}.)

When $\bf x, \bf y$ and ${\bf x}+{\bf y}$ are all primitive  the commutator $[P_{\bf x},P_{\bf y}] $ may be regarded as the difference of a switch of the curves $\bf x$ and $\bf y$, with $P_{{\bf x}+{\bf y}}$ appearing as the simultaneous smoothing at the $k$ crossings.

The proof of  Theorem \ref{thm_commutationrelations} relies   on a result in \cite{MM08} to establish that $$[P_{(m,0)},P_{(0,1)}] =(s^m-s^{-m})P_{(m,1)}.$$ Direct skein manipulation shows that $$[P_{(0,-1)},P_{(m,1)}] =(s^m-s^{-m})P_{(m,0)},$$ using Aiston's representation of $P_m$. The full theorem is  proved  from these two cases in Section \ref{sec_relations} using induction on $\det[ \xx \ \yy]$. Using a theorem of Przytycki in \cite{Prz92}, we prove the following (see Corollary \ref{cor_basis}).
\begin{theorem_noname}\label{thm_intropres}
As an abstract algebra, $H(T^2)$ has a presentation with generators $P_\xx$ for $\xx \in \Z^2$ and relations given by equation (\ref{eq_commintro}).
\end{theorem_noname}

 In general, if $M$ is a $3$-manifold, then $H(M)$ is a module over the algebra $H(\partial M)$. In particular, the skein module $H(S^3 \setminus K)$ of a knot complement is a module over the algebra $H(T^2)$. In the case where $K$ is the trivial knot we use earlier work of Morton and coauthors \cite{MH02, HM06} to give an explicit description of $\cc := H( S^1\times D^2)$ as a module over $H(T^2)$. In particular, as a module over the (commutative) `horizontal' subalgebra generated by $\{P_{m,0}\mid m \in \Z\}$ it is simultaneously diagonalizable with distinct eigenvalues. Over the `vertical' subalgebra generated by $\{P_{0,n}\mid n \in \Z\}$ the module $\cc$ is free of rank 1. We give explicit formulas for this action in Theorem \ref{thm_Cactionfull}.

These theorems can be viewed as analogues of the work of Frohman and Gelca in \cite{FG00}. They give a presentation of the Kauffman bracket skein algebra $K_q(T^2)$ of the torus $T^2$ and describe its action on the (Kauffman bracket) skein of the annulus. It turns out that the algebra $K_q(T^2)$ is the $t=q$ specialization of the $\sl_2$ spherical double affine Hecke algebra introduced by Cherednik in \cite{Che95}, which is an algebra that has attracted much attention in representation theory. We show that an analogous statement is true for the Homfly skein algebra $H$.

\subsection{The elliptic Hall algebra}
Let $X$ be a smooth elliptic curve over a finite field. In \cite{BS12}, Burban and Schiffmann gave an explicit presentation for the \emph{elliptic Hall algebra} $\E_{q,t}$, which is the Drinfeld double of the Hall algebra of the category $Coh(X)$ of coherent sheaves over $X$. They show that the structure constants are Laurent polynomials in $q,t$, where $q^2, t^2$ are the eigenvalues of the Frobenius operator on the first $l$-adic cohomology group of $X$. We may therefore view $q,t$ as formal parameters.

It turns out that $\E_{q,t}$ is closely related to a number of algebras that have been studied recently under different names. For example, this algebra (or one of its close cousins\footnote{By `cousin' we mean either the `positive half' $\E^+_{q,t}$ or a central extension of $\E_{q,t}$.}) has appeared as the following: 
\begin{itemize}
\item a generalized quantum affine algebra in \cite{DI97},
\item `a $(q,\gamma)$ analog of the $W_{1+\infty}$ algebra' in \cite{Mik07}
\item the `shuffle algebra' of \cite{FT11},
\item the `spherical $\gl_\infty$ double affine Hecke algebra' in \cite{SV11}, (see also \cite{FFJMM11a})
\item the `quantum continuous $\gl_\infty$' in \cite{FFJMM11a}
\item an algebra of operators on $\oplus K^T(\mathrm{Hilb}_n(\C^2))$ in \cite{SV13}. (See also \cite{FT11}, \cite{FFJMM11a}.)
\end{itemize}
As a consequence of Theorem \ref{thm_intropres}, we add another algebra to this list (see Theorem \ref{thm_HisotoE}). 
\begin{theorem_noname}\label{thm_introiso} 
The algebra $H$ is isomorphic to $\E_{q,q}$. In particular, any knot  $K \subset S^3$ provides a module over the algebra $\E_{q,q}$. 
\end{theorem_noname}

This theorem may seem surprising because the definitions of $\E_{q,t}$ and $H$ seem completely unrelated. One rough heuristic explanation for this isomorphism is as follows. When all parameters are set equal\footnote{``Setting all parameters equal to one'' is a statement that requires some care to be made precise, but we will not discuss this.} 
to $1$, the Homflypt skein algebra of a surface $F$ surjects onto $\O_{F,N} := \O(\chr(\pi_1(F), \GL_N(\C))$, the ring of functions on the scheme parameterizing representations of $\pi_1(F)$ up to equivalence. 
The $\gl_N$ spherical double affine Hecke algebra $\SH_{q,t}^N$ is an algebra depending on 2 parameters $q,t \in \C^*$, and when $q=t=1$, this algebra is isomorphic to $\O_{T^2,N}$. Finally, Schiffmann and Vasserot show in \cite{SV11} that $\E_{q,t}$ surjects onto $\SH_{q,t}^N$ for any $N$. Summarizing, when the parameters are set to 1, both $H$ and $\E_{1,1}$ surject onto the (commutative) algebras $\O_{T^2,N}$ for any $N$.

We also remark that Theorem \ref{thm_introiso} seems analogous to Kontsevitch's homological mirror symmetry for an elliptic curve (see \cite{PZ98} for a precise statement and proof in this case). Very roughly, in this case mirror symmetry predicts that the derived category $D^b(Coh(X))$ of coherent sheaves over an elliptic curve $X$ over $\C$ is 
equivalent to the Fukaya category of the symplectic torus $S^1 \times S^1$. Objects in this Fukaya category are simple closed curves `decorated' by a representation of $U(n)$, and under this equivalence, a sheaf on $X$ of slope $m/n$ is sent to the $(m,n)$ curve on $S^1\times S^1$, decorated by a certain representation. This is reminiscent of the isomorphism in Theorem \ref{thm_introiso}; however, it seems that there are significant difficulties (at best) in attempting to prove this isomorphism using mirror symmetry.

\subsection{Algebraic knots}
We next discuss an application of the isomorphism $H \cong \E_{q,q}$. If $K$ is an iterated cable of the unknot, then it is straightforward to give a \emph{cabling formula} for the $\lambda$-colored Homflypt polynomial $J^H(K,\lambda)$ of $K$ in terms of the action of the algebra $H$ on $\cc$ and the Homflypt evaluation map $\ev^H: \cc \to H(S^3)$ (see Proposition \ref{prop_cabling}). Using the isomorphism of Theorem \ref{thm_HisotoE} and a theorem in \cite{SV13}, we show that all objects used in the cabling formula have $t$-deformations. This allows us to define 3-variable polynomials $J^\E(K, \lambda)$ (see Definition \ref{def_JE}) that specialize to the Homflypt polynomial $J^H(K,\lambda)$. (Technically, these are rational functions for generic $t$ - see Remark \ref{rmk_rational}.)

In \cite{Sam14} similar polynomials were defined using the $\sl_2$ spherical double affine Hecke algebra, and Cherednik and Danilenko generalized these to arbitrary $\g$ shortly afterwards in \cite{CD14} (with a slightly simpler construction). For $\g = \gl_N$, we recall their definition of the polynomials $J^N(K,\lambda)$ in Definition \ref{def_JN}. We prove the following
in Theorem \ref{thm_iteratedcable}:
\begin{theorem_noname}
If $K$ is an iterated cable of the unknot, we have the following specializations:
\begin{eqnarray*}
J^\E(K,\lambda; q,t,u)\Big|_{q=s^{-2}, t=s^{-2}, u = v^2} &=& v^\bullet s^\bullet J^H(K,\lambda; v,s) \\
{ }  u^\bullet J^\E (K,\lambda; q,t,u)\Big|_{u = t^N} &=& q^\bullet t^\bullet J^N(K,\lambda; q,t)
\end{eqnarray*}
(where the powers denoted by ``$\bullet$'' depend on $K$ and $\lvert \lambda \rvert $, but not on $N$). In particular, the Connection Conjecture \cite[Conj.\ 2.4(i)]{CD14} is true.
\end{theorem_noname}

We remark that the definition of $J^\E(K)$ depends \emph{a priori} on a choice of presentation of $K$ as an iterated cable of the unknot. It isn't clear whether different choices for this presentation produce the same polynomials. (However, it was shown that certain different choices do produce the same polynomials $J^N(K)$ in \cite[Thm.\ 2.1(i)]{CD14}.)

At this point, one natural question is whether the modules over $\E_{q,q}$ associated to knots can be deformed to modules over $\E_{q,t}$ for any $t \in \C^*$ (in the style of \cite{BS14}). A positive answer may be interesting even for the module $\cc$ associated to the unknot. A central extension $\E_{q,t}^c$ of $\E_{q,t}$ acts on $\Lambda$, which is the `Fock space' with basis given by the set of all partitions (see, e.g., \cite{FFJMM11a} or \cite{SV13}). The module $\cc$ has a natural basis indexed by \emph{pairs} of partitions, and so can be identified with $\Lambda \otimes_\C \Lambda$ as a vector space. It does not seem that a module of this `size' has appeared in the recent literature about the representation theory of $\E_{q,t}$, but it does seem reasonable to expect that $\cc$ deforms. We hope to address this question in future work.

\subsection{Summary}
We now summarize the contents of the paper. In Section \ref{sec_background} we provide brief background and definitions. We then give a presentation for the algebra $H$ in Section \ref{sec_relations}. The module $\cc$ associated to the solid torus is described explicitly in Section \ref{sec_solidtorus}. In Section \ref{sec_ehall} we prove that $H$ is the $t=q$ specialization of the elliptic Hall algebra $\E_{q,t}$. In Section \ref{sec_adaptations} we describe different specializations of the Homflypt skein relations and their relation to other knot invariants. In Section \ref{sec_iteratedcables} we use the elliptic Hall algebra to construct a 3-variable polynomial that specializes to the Homflypt polynomial for iterated torus knots and to the 2-variable polynomials for $\gl_N$ constructed in \cite{CD14}.

\noindent\textbf{Acknowledgments:} We would like to thank I. Cherednik for helpful discussions about \cite{CD14}, and D. Muthiah for help with using \texttt{Sage} and for several enthusiastic conversations. We also thank Y. Berest, F. Bergeron, A. Oblomkov, V. Shende, O. Schiffmann, and E. Vasserot for enlightening discussions of their work and/or the present paper. The authors also benefited from the Research in Pairs program in Oberwolfach, where final editing on this paper was completed.

\section{Background}\label{sec_background}
In this section we first establish notation for `quantum numbers' and partitions, and we then give definitions and background for Homflypt skein modules.

\subsection{Notation}\label{sec_notation}
In this paper we will use the coefficient ring $R = \C[ v^{\pm 1},s^{\pm 1}, (s-s^{-1})^{-k}]$ (where $k$ ranges over $\N$), and we use the following `quantum numbers:'
\[ [d] := \frac{s^d-s^{-d}}{s-s^{-1}},\quad \quad \{d\} := s^d-s^{-d},\quad \quad \{d\}^+ := s^d + s^{-d}
\]

Let $\lambda = (\lambda_1,\cdots,\lambda_k)$ be a partition of length $k$. We will represent partitions by Young diagrams using the continental convention, so that a nonempty partition always has a box in its lower left corner. We will write $\lambda$ both for a partition and for its representation as a Young diagram. If $x$ is a box in $\lambda$ in row $i$ and column $j$, we will use the standard notations
\begin{equation}\label{eq_pnotation}
 a(x) = \lambda_i - j,\quad l(x) = \lvert \{n \mid \lambda_n \geq j\}\rvert - i,\quad a'(x) = j,\quad l'(x) = i
\end{equation}
Here $a(x)$, $l(x)$, $a'(x)$, and $l'(x)$ are the arm length, leg length, coarm length, and coleg length of $x$, respectively. Graphically, they are the number of cells strictly to the right, strictly above, weakly to the left, and weakly below $x$, respectively. (For a picture see, e.g. \cite[Fig.\ 1]{SV13}.) We will also use the \emph{content} $c(x)$ and \emph{hook length} $hl(x)$, which are defined by
\begin{equation}
 c(x) := j - i,\quad hl(x) := a(x) + l(x) + 1
\end{equation}

\subsection{Homflypt skein modules}
Let $M$ be an oriented 3-manifold. A \emph{framed oriented link} in $M$ is (an ambient isotopy class of) a smooth embedding $\sqcup (S^1 \times [0,1]) \hookrightarrow M$, where each copy of $S^1$ is oriented. Let $\mathcal L(M)$ be the free $R$-module spanned by framed oriented links in $M$, and let $\mathcal L'(M) \subset \mathcal L(M)$ be the $R$-submodule generated by the skein relations in Figure \ref{fig_skeinrelations}. (A \emph{skein relation} is a formal linear combination of links differing only inside a 3-ball as shown in the figure.)

\begin{center}
\begin{figure}[ht]
\bc
$\Xor - \Yor=(s-s^{-1})\ \Ior$ \qquad (Switch and smooth)

$\Rcurlor=v^{-1}\ \Idor\ ,\qquad \Lcurlor =v\ \Idor$ \qquad (Framing change)
\ec
\caption{The Homflypt skein relations}\label{fig_skeinrelations}
\end{figure}
\end{center}

\begin{definition}\label{def_homflyskeinmod}
The \emph{framed Homflypt skein module} $H(M)$ of the manifold $M$ is the quotient $\mathcal L(M) / \mathcal L'(M)$.
\end{definition}

In general, the Homflypt skein $H(M)$ has four important properties:
\begin{enumerate}
\item The skein $H(M)$ is graded by the homology group $\textrm{H}_1(M)$, since each skein relation involves only links in the same homology class.
 \item\label{item_mult} If $M = F \times [0,1]$, then $H(M)$ is an algebra (which is typically noncommutative). The product is given by stacking links (see Section \ref{sec_diagrams} for an example). The grading is additive under the product.
 \item If $\partial M = F$, then $H(M)$ is a module over the algebra $H(F \times [0,1])$. The action is given by ``pushing links from the boundary into $M$.'' If $a \in H(F \times [0,1])$ and $m \in H(M)$ are homogeneous, then $deg(a\cdot m) = \iota(deg(a)) + deg(m)$, where $\iota_*: \textrm{H}_1(F) \to \textrm{H}_1(M)$ is the map on homology induced by the inclusion $\iota:F \hookrightarrow M$.
 \item An oriented embedding $f:M \hookrightarrow N$ induces an $R$-linear map $f_*:H(M) \to H(N)$. When $f$ is a homeomorphism the map $f_*$ is an isomorphism.
\end{enumerate}

\begin{definition}
Write $H := H(T^2)$ for the (framed Homflypt) skein of the torus $T^2$, and   $\cc := H(S^1 \times I^2)$ for the skein  of the solid torus, with a choice of explicit homeomorphism from the solid torus to $(S^1\times I)\times I$ to specify $\cc$ as an algebra (see the beginning of Section \ref{sec_solidtorus}). 
\end{definition}

In this paper we study the algebra structure of $H$ and the $H$-module structure of $\cc$. The algebra $H$ is graded by $\xx \in \Z^2= H_1(T^2)$. Set \[H=\oplus_{\xx\in\Z^2} H_{\xx}\] where the degree of a link is given by its homology class.  Similarly the algebra $\cc$ has a $\Z$ grading. There is an action of $\SL_2(\Z)$ by algebra automorphisms on $H$ induced by the mapping class group action on $T^2$.

\subsection{Diagrammatic representations of links in $T^2$.}\label{sec_diagrams}

Use the classical presentation of $T^2$ as a square with pairs of edges identified, 
\bc \Torus
\ec as indicated by the coloured edges in the diagram above.  Then  a link in $T^2$ can be drawn as a diagram in the square with some arcs meeting the coloured edges in matching pairs.

For example here are  diagrams in $T^2$ of a 2-component link
\bc $D\quad =\quad$\torusD\ec    and a 1-component link \bc $E\quad =\quad$\torusE\ec along with diagrams for the products 
\bc$DE\quad=\quad$\torusDE\  and \ $ED\quad=\quad$\torusED\ .\ec

Elements in $H$ are linear combinations of diagrams, so for example  the commutator $[D,E]$ is the element $DE-ED\in H$.\\[2mm]

Using the convention that the square has sides along the usual axes we can draw the embedded $(m,n)$ curve in $T^2$ as \bc
$P_\xx$\quad =\quad  \minusonek \ec meeting the vertical edges in $m$ points and the horizontal edges in $n$ points.  Here  $(m,n) =(-1,k)=\xx\in\Z^2$ is the homology class of the curve, so that $P_\xx$ lies in the graded subspace $H_\xx\subset H$. The curve will cross the edges of the square according to the signs of $m$ and $n$; in this diagram $m=-1$ and $n=k>0$.

This defines elements $P_\xx$ for each primitive $\xx\in\Z^2$, in other words where $m$ and $n$ are coprime. As a simple example we have
\bc $ P_{1,0}$ \ = \ \Ponezero, \quad  $P_{0,1} $\ = \ \Pzeroone \ .\ec Their commutator then satisfies
\beqn [P_{1,0},P_{0,1}]&=& \Ppos\  -\  \Pneg\\[1mm]
&=& (s-s^{-1})\ \Poneone\\[1mm]
&=&(s-s^{-1} )P_{1,1},
\eeqn using the switch and smooth skein relation.

Clearly if the embedded curve $P_\xx$ crosses the embedded curve $P_\yy$ once in the positive sense then the switch and smooth relation applied at the crossing results in the embedded curve $P_{\xx+\yy}$ and their commutator satisfies the equation
\[ [P_\xx,P_\yy]=(s-s^{-1})P_{\xx+\yy}.\]

\begin{remark} The signed number of crossings of curves with homology $\xx$ and  $\yy$ is equal to the determinant $\det[\xx\ \yy]$ of the $2\times 2$ matrix with columns $\xx$ and $\yy$.
\end{remark}

Before giving the definition of the elements $P_{(m,n)}\in H$ where $m$ and $n$ are not coprime we use  diagrams in the square meeting the edges in matched pairs to represent elements of some other skeins. 

\subsection{The skein $\cc$ of the annulus}\label{sec_annulus}
If we identify the two horizontal edges only we get an annulus. Elements of its skein $\cc$ can be represented by diagrams which do not meet the vertical edges.  Any such diagram also gives a diagram in $T^2$ by identifying the two vertical edges. The elements arising in this way in $H$ will lie in the graded part $H_{(0,k)}$ where $k$ is the signed  number of crossings with the horizontal edge.  

For $i,j \ge0$, and $i+j+1=k>0$ define elements $A_{i,j}\in \cc$ diagrammatically by\\[1mm]

\bc $A_{i,j}$\quad=\quad \labellist\small
\pinlabel {$i$} at 165 410
\pinlabel {$j$} at 275 410
\endlabellist\aijannulus \ .\\[4mm]
\ec
Set \[X_k:=\sum_{i=0}^{k-1} A_{i,j}\in\cc\]

Aiston \cite{Ais96} showed that $X_k$ represents a multiple of the $k$th power sum when interpreting $\cc $ in terms of symmetric functions. Further work by Aiston, Morton and subsequent collaborators has identified other nice algebraic and skein theoretic properties of $X_k$ and more particularly the exact power sum $P_k:=(s-s^{-1})/(s^k-s^{-k}) X_k$, which we shall also use in this paper.

\begin{definition}\label{def_px}
For $k > 0$ we define $P_{0,k}$ to be the result of decorating the $(0,1)$ curve in $T^2$ by $P_k$ from $\cc$, or equivalently \[P_{0,k}=\frac{s-s^{-1}}{s^k-s^{-k}}\sum_{i=0}^{k-1} A_{i,j}\] after identifying the vertical edges as well in the square to give diagrams in $T^2$.  

To define $P_\xx$ in general write $\xx=k(m,n)$ with $m$ and $n$ coprime and $k>0$. Then \emph{decorate} the embedded $(m,n)$ curve by the element $P_k\in\cc$ from the skein of the annulus to specify $P_\xx$.
\end{definition}

\subsection{Decoration and framing}
The term \emph{decoration}   in our definition above  has been used widely by Morton and others in describing satellite invariants of framed knots and links  in manifolds. 

What is meant by \emph{decorating} a framed oriented curve $K$ in $M$ by a framed curve $Q$ in $S^1\times I$  is to embed the thickened annulus $S^1\times I$ on the neighbourhood of $K$, respecting orientation \emph{and} framing. The image of $Q$ in $M$ is   the satellite of $K$ with pattern $Q$, and is referred to as $K$ \emph{decorated by} $Q$. It carries a framing which is determined by the framing of $Q$.

The embedding from $S^1\times I$ to $M$ depends on the framing and orientation of $K$.  It induces a linear map of skeins from $\cc= H(S^1\times I)$ to $H(M)$, with image denoted by $\cc_K$.  

In the case where $M=F\times I$ and $K$ is an \emph{embedded} curve in the surface $F$, which is framed by its neighbourhood in $F$, the induced linear map is an algebra homomorphism.  The image $\cc_K$ is then a subalgebra of $H(F)$.

When $F=T^2$ there is an embedded curve $P_\xx$ for each \emph{primitive} $\xx\in\Z^2$. We have  defined $P_{k\xx}$ above to be the image of $P_k\in\cc$ under the homomorphism $\cc\to H$ determined by the framed embedded curve $P_\xx$.  

\begin{remark}
When $k=1$ the result $P_\xx$ is just the oriented $\xx = (m,n)$ curve, as defined earlier. For $P_{0,n}$ with $n<0$ we decorate the $(0,-1)$ curve by $P_{-n}$. Also, since $SL(2,\Z)$ acts on $T^2$ preserving the framing annuli of embedded curves we see that its action on $H$ permutes the elements $P_\xx$.
\end{remark}
 
\subsection{Hecke algebras of type $A$}
There is a family of algebras modelled by the skein of the square with diagrams meeting the horizontal edges in matching pairs, but not meeting the vertical edges.  Where there are $n$ meeting points on each edge, all oriented upwards, this skein models the Hecke algebra $H_n$ as in \cite{MT90} and \cite{AM98}. With $n$ upwardly  and $p$ downwardly oriented meeting points the skein is the mixed algebra $H(n,p)$ in \cite{KM93} and in \cite{MH02}.  In both cases the algebra product comes from placing one square above another.
  
 \subsection{Affine Hecke algebras of type $A$}
 
Use the square with  the vertical edges identified, to give a cylinder, and take diagrams meeting the horizontal edges in $n$ matching pairs, oriented upwards, to give an algebra closely related to the affine Hecke algebra $\dot{H}_n$. 

A description of the affine Hecke algebra in this way is given by Graham and Lehrer \cite{GL03}. They restrict the diagrams to be braids in the cylinder, without closed curves, and with all strings running monotonically in the vertical direction. They then have  no need for the parameter $v$, nor the framing change relation.  In their work the regular elementary braids $\sigma_i$ represent the elements conventionally called $T_i$ in presentations of $\dot{H}_n$. They give cylindrical braid representatives for the commuting elements in $\dot{H}_n$ known as $X_i$, and for the element $\tau$ which realizes a $1/n$ turn around the cylinder.

 Because we are using $X_i$ in this paper in a different context, following the terminology of \cite{MM08}, we 
 refer here to   Graham and Lehrer's braids in red. Their elements $\r{X_i}, i=1,\ldots,n$ are
shown here.
\begin{figure}[ht]\label{Murphycylinder}
\bc $\r{X_i}\quad= \quad $\labellist\small
\pinlabel {$i$} at 210 335
\endlabellist \Murphycylinder\\[4mm]
\ec
\caption{The commuting elements $\r{X_i}$ in the affine Hecke algebra $\dot{H}_n$.}
\end{figure}

Symmetric functions in these are central elements in $\dot{H}_n$. In this context, it is  worth extending to the full skein of the cylinder without restricting to the use of braids. In the full skein there is a nice model of the sum $\sum \r{X_i}$ as a sort of commutator of the identity braid with the closed curve $P_{1,0}$. The exact result is the equation
\[\backid\quad - \quad \frontid\quad =\quad (s-s^{-1})\sum_{i=1}^n \r{X_i }\ ,\] proved by applying the switch and smooth relations on the $n$ crossing points of the first diagram in sequence.  Both diagrams on the left hand side are then clearly central in the full skein.

\begin{remark} The full skein of the cylinder with $n=1$ is used by Lukac \cite{Luk01}, and by Morton and Hadji \cite{HM06}, under the name $\mathcal A$, in proving results about $\cc$.
One relevant result from \cite{MM08} is that by using $P_{k,0}$ in place of $P_{1,0}$ the commutator when $n=1$ is $s^k-s^{-k}$ times the cylindrical $1$-braid going $k$ times around the cylinder. This is illustrated below, with the cylinder drawn as an annulus, in Equation (\ref{cylindercommutator}). It leads immediately to the representation of $ (s^k-s^{-k})\sum_{i=1}^n \r{X_i}^k$ for every $n>1$.

Very close analogues of $\r{X_i}$ in $H_n$ are used in \cite{Mor02, Mor02b} where they represent the Murphy operators. The notation $T_i$ is used in that paper, where a similar construction gives a quick proof that  symmetric polynomials in the commuting elements $T_i$ are central in $H_n$.
\end{remark}

 It is also worth considering the algebras $\dot{H}_{n,p}$ where the vertical edges are again identified to give a cylinder, and diagrams meet the horizontal edges in $n$ upward and $p$ downward matched pairs of points. 
  By work of Turaev \cite{Tur88} the elements $X_m$ generate $H(S^1\times [0,1])$ as an algebra over $R$.

\section{A presentation for $H$}\label{sec_relations}
In this section we give a presentation of $H$ using the elements $P_\xx$ of Definition \ref{def_px}.

\begin{lemma}
 The algebra $H$ is generated by the elements $P_{\xx}$.
\end{lemma}
\begin{proof}
Any element of $H$ can be reduced to a sum of products of knots using the skein relations. Then a knot in $H$ is in some graded piece $H_{j,k}$, and it follows from \cite{Prz92} that the skein module of the annulus surjects onto any particular graded piece $H_{j,k}$ by embedding the annulus onto a neighborhood of the $(j/d, k/d)$ curve (where $d = gcd(j,k)$). The claim then follows from the fact that the $X_m$ generate the skein module of the annulus, which was proved by Turaev in \cite{Tur88}, and the fact that over $R$ the $X_m$ can be written in terms of the $P_\xx$.
\end{proof}

\begin{theorem}\label{thm_commutationrelations}
The elements $P_\xx$ for $\xx \in \Z^2$ satisfy the following commutation relation:
\begin{equation}\label{formula_allrelations}
  [P_\xx,P_\yy] = \{\mathrm{det}[\xx\, \yy]\} P_{\xx+\yy}
 \end{equation}
\end{theorem}
\begin{proof}
We  separate this proof into two subsections. In Section \ref{proof_ofsomerelations} we prove the following relations   using methods and results of \cite{MM08}:
\begin{align}
  \,[P_{k\xx}, P_{j\xx}] &= 0\notag\\
  \,[P_{1,0}, P_{-1,k}] &= \{k\} P_{0,k}\label{formula_somerelations}\\
  \,[P_{1,0}, P_{0,k}] &= \{k\} P_{1,k}\notag
 \end{align}
Then in Section \ref{proof_allrelations} we show that the relations (\ref{formula_somerelations}) imply (\ref{formula_allrelations}).  
\end{proof}

Before we prove this theorem, we use a theorem in \cite{Prz92} to show that this gives a presentation of $H$. More precisely, let $A$ be the abstract algebra generated by $P_{j,k}$ subject to the relations in Theorem \ref{thm_commutationrelations}. 

\begin{corollary}\label{cor_basis}
The natural surjection $A \twoheadrightarrow H$ is an isomorphism.
\end{corollary}
\begin{proof}
We begin by recalling a basis for $H$ constructed in \cite{Prz92}. We pick a linear ordering $\leq$ of elements in $\pi^0 := \Z^2 \setminus \{0\}$ according to the angle with the positive $x$-axis. Then for each $w = (m,n) \in \pi^0$ we pick a representative diagram $B_w$, which is defined to be the curve $A_{k,0}$ (defined in Section \ref{sec_annulus}) inserted on the $(m/k,n/k)$ curve on the torus, where $k = gcd(m,n)$. If  $Sym(\pi^0)$ is the symmetric algebra of $R\pi^0$, then Przytycki defines a map $Sym(\pi^0) \to H$ which takes a monomial $w_1 \cdots w_n$ to the product $B(w_1)B(w_2)\cdots B(w_n)$, where $w_1 \leq w_2 \cdots \leq w_n$. Then \cite[Thm.\ 6.2]{Prz92} states that this is an $R$-isomorphism.

We then note that over $R$, the diagram $B(w)$ can be written uniquely in terms of diagrams $P_{c w}$ for $c \in \N$. Then the commutation relations in Theorem \ref{thm_commutationrelations} allow one to order the $P_\xx$ according to the angle between $\xx$ and the positive $x$-axis. This completes the proof of the corollary.
\end{proof}

\subsection{Certain commutation relations}
\label{proof_ofsomerelations}
\begin{proof}  

The first equation of (\ref{formula_somerelations}) is obvious because the two elements in question lie in parallel annuli.

To prove the second equation, we work with  diagrams in the square with edges identified, as discussed above. The commutator $[P_{1,0}, P_{-1,k}]$ is  represented by the torus diagrams $$ \back\quad-\quad\front \ .$$  

 We introduce intermediate torus diagrams $C_{j,k-j}, j=0,\ldots.k$, in which the first $j$ crossings on the $(1,0)$ curve in $$\back$$ are switched from over to under. 

Thus   
\bc
\labellist\small
\pinlabel {$j$} at 140 315
\pinlabel {$k-j$} at 240 315
\endlabellist $C_{j,k-j}\quad =\quad $ \Cij\ .\\[6mm]\ec

Then  $C_{0,k}=P_{1,0} P_{-1,k}$  and $C_{k,0}=P_{-1,k}P_{1,0} $.

The skein relation applied at the circled crossing on $C_{j,k-j}$  gives
$$C_{j,k-j}-C_{j+1,k-j-i} = (s-s^{-1}) D_{j,k-j-1}.$$  
Here 
\bc
\labellist\small
\pinlabel {$j$} at 142 325
\pinlabel {$i$} at 250 325
\endlabellist $D_{j,i}\quad =\quad$ \smoothedtorus\\[6mm]\ec 
is the smoothed diagram with $j$ crossings on the string from the left-hand vertical edge  and $i$ crossings on the string to the right-hand vertical edge.

Then the commutator $[P_{1,0}, P_{-1,k}] $ is represented by
\begin{eqnarray*} C_{0,k}-C_{k,0} &=& \sum_{j=0}^{k-1} (C_{j,k-j}-C_{j+1,k-j-1})\\
&=& (s-s^{-1}) \sum_{j=0}^{k-1} D_{j,k-j-1}.
\end{eqnarray*}

Now the diagram $D_{j,i}$   in $T^2$ can be isotoped by moving $j$ crossings horizontally to the left to get an equivalent diagram which does not meet the vertical edge. The result is  the closed braid diagram \bc
\labellist\small
\pinlabel {$i$} at 162 420
\pinlabel {$j$} at 270 420
\endlabellist \aijtorus \ .\\[5mm] \ec This is the element $A_{i,j}\in\cc$   described earlier, running in the direction of the $(0,1)$ curve on $T^2$. This establishes that $[P_{1,0}, P_{-1,k}]$ is represented by the sum of closed braids.$$(s-s^{-1})\sum_{i+j=k-1} A_{i,j}$$ following the $(0,1)$ curve in $T^2$.

Now in \cite{Mor02}  it is shown that $$\sum_{i+j=k-1} A_{i,j}= (s^k-s^{-k})/(s-s^{-1}) P_k,$$  and the same result is repeated in \cite{MM08}, when the closed braids $A_{i,j}$ are represented as the closure of different but
conjugate braids. The element $P_k$ decorating the $(0,1)$ curve in $T^2$ represents $P_{0,k}$ in our current notation, and this proves the equation
 $$[P_{1,0}, P_{-1,k}] =\{k\}P_{0,k}.$$
 \\[2mm]
 
To prove the final equation of (\ref{formula_somerelations}), we first note that a diagram on $T^2$ cut open along a $(0,1)$ curve gives a diagram in the annulus with some matched point on the two boundary curves. The product $P_{1,0} P_{0,k}$ can thus be represented  by the diagram
 
 \bc
{\labellist\small
\pinlabel {$P_k$} at 60 140
\endlabellist}  \Ahe \ec in the annulus with one input on the left boundary and one output on the right boundary. The commutator $[P_{1,0},P_{0,k}]$ is then represented in the annulus by
 \bc {\labellist\small
\pinlabel {$P_k$} at 60 140
\endlabellist} \Ahe\ $\  -\ $ {\labellist\small
\pinlabel {$P_k$} at 60 140
\endlabellist}  \Aeh\ec

 We then apply directly Theorem 4.2 of \cite{Mor02b} which gives an equation in the skein of the annulus with one  point on each boundary curve. In diagrammatic form, Theorem 4.2 of that paper shows that
 
\begin{equation} {\labellist\small
\pinlabel {$P_k$} at 60 140
\endlabellist} \Ahe\ - {\labellist\small
\pinlabel {$P_k$} at 60 140
\endlabellist}  \Aeh\ = (s^k-s^{-k})\  \Aaa\ .\label{cylindercommutator}\end{equation}
 
The curve in the right-hand diagram circles the annulus $k$ times, here shown in the case $k=2$.  When the annulus  follows the $(0,1)$ curve in $T^2$, and its boundary curves are rejoined to form $T^2$ the curve in the right-hand diagram becomes the $(1,k)$ curve in $T^2$  while the other two diagrams yield $P_{1,0} P_{0,k}$ and $P_{0,k} P_{1,0}$ respectively. This leads immediately to the equation
 $$[P_{1,0},P_{0,k}]=\{k\}P_{1,k}.$$

\end{proof}

\subsection{All commutation relations}
\label{proof_allrelations}
In this section we prove Proposition \ref{prop_allfromsome}, which shows that the equations (\ref{formula_allrelations}) follow from equations (\ref{formula_somerelations}). In what follows, we write $d(\xx, \yy) = \det\left[\xx\,\, \yy\right]$ for $\xx, \yy \in \Z^2$ and $d(\xx) = gcd(m,n)$ when $\xx = (m,n)$. We will also use the following terminology: 
\[
(\xx,\yy) \in \Z^2 \times \Z^2 \textrm{ \emph{satisfies}  (\ref{formula_allrelations}) if  } [P_\xx,P_\yy] = \{d(\xx,\yy)\}P_{\xx+\yy}
\]

The idea of the proof is to induct on the determinant of the matrix with columns $\xx$ and $\yy$. To induct, we write $\xx = \aab + \bb$ for carefully chosen vectors $\aab, \bb$ and then use the following lemma.

\begin{lemma}\label{lemma_trueforab}
 Assume $\aab + \bb = \xx$ and that $(\aab,\bb)$ satisfies (\ref{formula_allrelations}). Further assume that the four pairs of vectors $(\yy, \aab)$, $(\yy, \bb)$, $(\yy+\aab,\bb)$, and $(\yy+\bb,\aab)$ satisfy (\ref{formula_allrelations}). Then the pair $(\xx,\yy)$ satisfies (\ref{formula_allrelations}).
\end{lemma}
\begin{proof}
 By the first assumption, we have $[P_\aab, P_\bb] = \{d(\aab, \bb)\}P_\xx$. We then use the Jacobi identity and the remaining assumptions to compute
 \begin{eqnarray*}
  -\{d(\aab, \bb)\}[P_\xx,P_\yy] &=& -[[P_\aab,P_\bb],P_\yy]\\
  &=& [[P_\yy,P_\aab],P_\bb] + [[P_\bb,P_\yy],P_\aab]\\
  &=& \{d(\yy,\aab)\}[P_{\yy+\aab},P_\bb] + \{d(\bb,\yy)\}[P_{\bb+\yy},P_\aab]\\
  &=& \Big(\{d(\yy,\aab)\}\{d(\yy+\aab,\bb)\} + \{d(\bb,\yy)\}\{d(\bb+\yy,\aab)\}\Big) P_{\xx+\yy}\\
  &=:& cP_{\xx+\yy}
 \end{eqnarray*}
We now use the identity $\{m\}\{n\} = \{m+n\}^+ - \{m-n\}^+$ to simplify the coefficient $c$:
\begin{eqnarray*}
 c &=& \{d(\yy,\aab)\}\{d(\yy+\aab,\bb)\} + \{d(\bb,\yy)\}\{d(\bb+\yy,\aab)\} \\
 &=& \{d(\yy,\aab) + d(\yy,\bb) + d(\aab,\bb)\}^+ - 
      \{d(\yy,\aab) - d(\yy,\bb) - d(\aab,\bb)\}^+\\ 
 &\,& + \{d(\bb,\yy) + d(\bb,\aab) + d(\yy,\aab)\}^+ - \{d(\bb,\yy)-d(\bb,\aab) - d(\yy,\aab)\}^+\\
 &=& \{d(\yy,\xx) + d(\aab,\bb)\}^+ - \{d(\xx,\yy) - d(\bb,\aab)\}^+\\
&=& \{-d(\xx,\yy) + d(\aab,\bb)\}^+ - \{d(\xx,\yy) + d(\aab,\bb)\}^+\\
 &=& -\{d(\aab,\bb)\}\{d(\xx,\yy)\}
\end{eqnarray*}
This completes the proof of the lemma.
\end{proof}

We next prove the following elementary lemma (which is a slight modification of \cite[Lemma 1]{FG00}). This lemma is used to make a careful choice of vectors $\aab, \bb$ so that the previous lemma can be applied.

\begin{lemma}\label{lemma_diophantine}
 Suppose $p,q \in \Z$ are relatively prime with $0 < q < p$ and $p > 1$. Then there exist $u,v,w,z \in \Z$ such that the following conditions hold:
 \begin{eqnarray}
  u + w &=& p,\quad v + z = q\notag\\
  0 < u, w &<& p\label{equation_conditionsonuvwz}\\
  uz - wv &=& 1\notag
 \end{eqnarray}
\end{lemma}
\begin{proof}
 Since $p$ and $q$ are relatively prime, there exist $a,b \in \Z$ with $ bq - ap = 1$. This solution can be modified to give another solution $a' = a + q$ and $b' = b + p$, so we may assume $0 \leq b < p$. We then define 
 \[
  u=b,\quad v=a,\quad w=p-b,\quad z=q-a
 \]
 By definition, $u,v,w,z$ satisfy the first condition of (\ref{equation_conditionsonuvwz}), and the inequalities $0 \leq b < p$ and $p > 1$ imply the second condition. To finish the proof, we compute
 \[
  uz - wv = b(q-a) - a(p-b) = bq - ap = 1
 \]
\end{proof}

\begin{remark}\label{remark_gl2z}
 There is a natural $R$-linear anti-automorphism $\sigma:H \to H$ which ``flips $T^2 \times [0,1]$ across the $y$-axis and inverts $[0,1]$.'' In terms of the elements $P_{a,b}$, we have $\sigma(P_{a,b}) = P_{a,-b}$. We therefore have an \emph{a priori} action of $\mathrm{GL}_2(\Z)$ on $H$, where elements of determinant $1$ act by algebra automorphisms, and elements of determinant $-1$ act by algebra anti-automorphisms.
\end{remark}

\begin{proposition}\label{prop_allfromsome}
 Suppose $A$ is an algebra with elements $P_\xx$ for $\xx \in \Z^2$ that satisfy equations (\ref{formula_somerelations}). Furthermore, suppose that there is a $\GL_2(\Z)$ action by (anti-)automorphisms on $A$ as in Remark \ref{remark_gl2z}, and that the action of $\GL_2(\Z)$ is given by $\gamma(P_\xx) = P_{\gamma(\xx)}$  for $\gamma \in \GL_2(\Z)$.  Then the $P_\xx$ satisfy the equations (\ref{formula_allrelations}).
\end{proposition}
\begin{proof}
  The proof proceeds by induction on $\lvert d(\xx,\yy)\rvert $, and the base case $d(\xx,\yy) = \pm 1$ is immediate from Remark \ref{remark_gl2z} and one application of the skein relation. Now assume
 \begin{equation}\label{assumption1} 
\textrm{for all } \xx',\yy' \in \Z^2\textrm{ with } \lvert d(\xx',\yy')\rvert < d(\xx,\yy), \textrm{ we have } [P_{\xx'},P_{\yy'}] = \{d(\xx',\yy')\} P_{\xx' + \yy'}
 \end{equation}
We would like to show that $[P_{\xx},P_{\yy}] = \{d(\xx,\yy)\} P_{\xx + \yy}$. By Remark \ref{remark_gl2z}, we may assume
 \[
  \yy = \left(\begin{array}{c} 0\\r\end{array}\right),\quad \xx = \left(\begin{array}{c} p\\q\end{array}\right),\quad d(\xx) \leq d(\yy),\quad 0 \leq q < p
 \]
If $p=1$, then this equation follows from (\ref{formula_somerelations}), so we may also assume $p > 1$. \\[2mm]

\noindent \emph{Case 1:} Assume $0 < q$. 

Let $p' = p / d(\xx)$ and $q' = q / d(\xx)$ - by the assumption $0 < q$, we see that $d(\xx) < p$, so $p' > 1$. We can therefore apply Lemma \ref{lemma_diophantine} to $p',q'$ to obtain $u,v,w,z \in \Z$ satisfying 
\begin{equation}\label{assumption1.5}
 uz - vw = 1,\quad uq' - vp' = 1,\quad u + w = p',\quad v + z = q',\quad 0 < u,w < p'
\end{equation}
We then define vectors $\aab$ and $\bb$ as follows (the properties listed follow from (\ref{assumption1.5})):
\begin{equation}\label{assumption2}
 \aab := \left( \begin{array}{c}  d(\xx)u\\ d(\xx)v\end{array}\right),\quad 
 \bb := \left( \begin{array}{c} d(\xx)w\\ d(\xx)z\end{array}\right),
 \quad \aab + \bb = \xx,\quad d(\aab, \bb) = d(\xx)^2 
\end{equation}

Using Lemma \ref{lemma_trueforab} and Assumption (\ref{assumption1}), it is sufficient to show that the absolute values of each of $d(\aab, \bb)$, $d(\yy,\bb)$, $d(\yy,\aab)$, $d(\yy+\aab,\bb)$, and $d(\yy+\bb,\aab)$ are strictly less than $\lvert pr \rvert = d(\xx,\yy)$. First,
\[
d(\aab,\bb) = d(\xx)^2 \leq d(\xx)d(\yy) = d(\xx)r < pr 
\]
where the last inequality follows from the assumption $0 < q < p$. Second, the absolute values of $d(\yy,\bb)$ and $ d(\yy,\aab) $ are strictly less than $pr$ by the inequalities in (\ref{assumption1.5}). Third, we compute
\begin{eqnarray*}
 -d(\yy+\aab, \bb) &=& -d(\yy,\bb) - d(\aab, \bb)\\
 &=& d(\xx)wr - d(\xx)^2\\
 &<& d(\xx)wr\\
 &\leq& pr
\end{eqnarray*}
Finally, we compute
\begin{eqnarray*}
 -d(\yy+\bb, \aab) &=& -d(\yy,\aab) - d(\bb, \aab)\\
 &=& d(\xx)ur + d(\xx)^2\\
 &\leq& \left(d(\xx)u + d(\xx)\right)d(\yy)\\
 &=& (u+1)d(\xx)r
\end{eqnarray*}

Therefore, we will be finished once we show that the absolute value of $(u+1)d(\xx)$ is strictly less than $p$. We now split into subcases:\\[2mm]

\noindent\emph{Subcase 1a:} If $u + 1 < p'$, then $(u+1)d(\xx)r < p'd(\xx)r = pr$, and we are done. 

\noindent \emph{Subcase 1b:} Assume $u + 1 = p'$. By equation (\ref{assumption1.5}), we have 
 \[
1 = uq' - vp' = (p'-1)q' - vp'  \implies p'(q'-v) = 1 + q' < 1 + p'
 \]
Since $p' > 1$, the last inequality implies $q' - v = 1$, which implies $v = q'-1$ and $z = 1$. Since $uz-vw = 1$, this implies $(p'-1)-(q'-1) = 1$, which implies $q' = p'-1$. If we write $g = d(\xx)$, we then have
\[
 -d(\yy + \bb, \aab) = -\det\left[ \begin{array}{cc}g&p -g\\g + r&p-2g\end{array}\right] =  rp + g(g-r) = rp + d(\xx)(d(\xx)-r)
\]
We already assumed that $d(\xx) \leq r$, and if this inequality is strict, then we are done.

\noindent\emph{Subcase 1c:} In this subcase, we are reduced to proving that the vectors $\yy = (0,r)$ and $\xx = (rp', rp'-r)$ satisfy (\ref{formula_allrelations}). First, suppose $r = 1$. Then there is a matrix in $\SL_2(\Z)$ that fixes $\yy$ and sends $\xx \mapsto (p', -1)$. Therefore, if $r=1$, the second equation of (\ref{formula_somerelations}) implies that $(\xx,\yy)$ satisfies (\ref{formula_allrelations}).

Now we assume $r > 1$ and replace our previous choice of $\aab$ and $\bb$ with a choice which is better adapted to this particular subcase. We define
\[
 \aab := \left(\begin{array}{c} 0\\-1 \end{array}\right),\quad \bb := \left(\begin{array}{c} rp'\\rp' - r + 1 \end{array}\right)
\]
Since $r > 1$, it is clear that the absolute values of the determinants of the matrices $[\aab , \bb]$, $[\yy,\aab]$, $[\yy+\aab,\bb]$, and $[\yy+\bb,\aab]$ are all less than $r^2p'$. This together with Assumption (\ref{assumption1}) shows that these pairs of vectors satisfy (\ref{formula_allrelations}). Finally, since $rp'-r$ and $rp'-r+1$ have different parities, we see that $d(\bb) \not= d(\xx) = r$. This together with Subcase 1b shows that $(\yy,\bb)$ satisfies (\ref{formula_allrelations}). Then Lemma \ref{lemma_trueforab} finishes the proof of this subcase, which finishes the proof of Case 1. \\[2mm]

\noindent \emph{Case 2:} In this case we assume $q=0$. We define $\aab, \bb$ similarly to Subcase 1c, so we have
\[
 \yy = \left(\begin{array}{c} 0\\r\end{array}\right),\quad \xx = \left(\begin{array}{c} p\\0\end{array}\right),\quad \aab := \left(\begin{array}{c} 0\\-1 \end{array}\right),\quad \bb := \left(\begin{array}{c} p\\ 1 \end{array}\right)
\]
If $r=1$, then the third equation of (\ref{formula_somerelations}) implies that the pair $(\xx,\yy)$ satisfies (\ref{formula_allrelations}). If $r > 1$, then an identical argument to Subcase 1c finishes the proof of this case and of the theorem.
\end{proof}

\section{The skein of the annulus as a module over the algebra $H$.}\label{sec_solidtorus}

We now describe the action of the algebra $H$ on the skein  $\cc$ of the annulus. 
Draw the torus $T^2$ as the boundary of a standardly embedded solid torus $V\subset\RR^3$. Once an orientation and a framing for the core of $V$ have been chosen, by choosing an oriented annular neighbourhood of the core curve, we can regard $V$ as a thickened annulus and get an explicit identification of the skein  $H(V)$ of $V$ with the skein $\cc$.  In the diagram below
\bc
\CylinderAnnulus
\ec
we indicate the relation of the torus and the framed core of $V$. 

Parametrize $T^2$ so that the core lies in the direction of the $(0,1)$ curve,  and the  framing of the core is chosen to agree with the neighbourhood framing of the $(0,1)$ curve in $T^2$. The $(1,0)$ curve in $T^2$ is then a meridian of the solid torus $V$.

As remarked in Section 2, the skein $H(V)\cong \cc$ is a module over the algebra $H(\partial V) \cong H$.  To describe  $h\cdot c\in\cc$ explicitly for $h\in H$ and $ c\in\cc$, we represent $c$ by a framed diagram in the core annulus, and $h$ by a framed diagram in $\partial V =T^2$. The union of these two framed diagrams in $V$ then represents $h\cdot c$ in $H( V) =\cc$. (Properly $h$ and $c$ are represented by some $R$-linear combination of framed diagrams, and we extend the construction bilinearly).

\subsection{The action}
 We now give precise statements about the action of $H$ on $\cc$. Because the core annulus of $V$ is parallel to the $(0,1)$ curve in $T^2$ the action of $P_{0,n}$ on $\cc$ is simply multiplication by $P_n$ in $\cc$. 

  By results of Morton and Hadji \cite{MH02,HM06}, the module $\cc$ has an $R$-linear basis given by elements $Q_{\lambda,\mu}$, where $\lambda,\mu$ range over the set of partitions. The basis elements $Q_{\lambda,\mu}$ are shown there to be eigenvectors  of the `meridian maps' from $\cc$ to itself defined for each $m$ by $c\mapsto  P_{m,0}\cdot c$. We will describe the action of $H$ in this basis. Then at the end of the section we will collect from the literature several facts about the $Q_{\lambda,\mu}$ for the reader's convenience. 

For   partitions $\lambda, \mu$   write 
\[
 s_{\lambda,\mu} := (s-s^{-1})\left( v^{-1} \sum_{x \in \lambda} s^{2c(x)} - v\sum_{x \in \mu} s^{-2c(x)}\right) + \frac{v^{-1}-v}{s-s^{-1}} 
\]
where $c(x) = j - i$ is the \emph{content} of the cell $x$ in row $i$ and column $j$. We will use the continental convention for Young diagrams, so that the unique partition of $1$ corresponds to a cell in the lower-left corner of the diagram, which is in row and column $0$.

For a partition $\lambda$ define its \emph{content polynomial} $C_\lambda(t)\in \Z[t^{\pm1}]$ by
\[C_\lambda(t)=\sum_{x\in\lambda} t^{c(x)}.\]
Then $s_{\lambda,\mu} = (s-s^{-1})(v^{-1} C_\lambda(s^2)- vC_\mu(s^{-2}))+ \frac{v^{-1}-v}{s-s^{-1}}$ 

\begin{theorem}\label{thm_Caction}
In the basis $Q_{\lambda,\mu}$ of $\cc$ described in \cite{MH02, HM06}, the action of $H$ is determined by the equations
 \begin{eqnarray*}
  P_{1,0}\cdot Q_{\lambda,\mu} &=& s_{\lambda,\mu}Q_{\lambda,\mu}\\
  P_{-1,0}\cdot Q_{\lambda,\mu} &=& s_{\mu,\lambda}Q_{\lambda,\mu}\\
  P_{0,1}\cdot Q_{\lambda,\mu} &=& \sum_{\alpha \in \lambda +1} Q_{\alpha,\mu} + \sum_{\beta \in \mu -1}Q_{\lambda,\beta}\\
  P_{0,-1}\cdot Q_{\lambda,\mu} &=& \sum_{\beta \in \lambda -1} Q_{\beta,\mu} + \sum_{\alpha \in \mu +1}Q_{\lambda,\alpha}
  \end{eqnarray*}
where $\lambda +1,\lambda -1$ are the set of partitions where one cell has been added to (subtracted from) $\lambda$, respectively.
\end{theorem}
\begin{proof}
 The first two equations are  the meridian map formulae in \cite[Thm.\ 3.9]{HM06}, and the last two are the product formulae in \cite{MR10}.
\end{proof}

\begin{lemma}
 The statements in Theorem \ref{thm_Caction} completely determine the module structure of $\cc$.
\end{lemma}
\begin{proof}
 Since the $Q_{\lambda,\mu}$ form a basis for $\cc$ the action of  $P_{\pm 1,0}$ and $P_{0,\pm 1}$ on $\cc$ is completely described by Theorem \ref{thm_Caction}. Over $R$ these elements generate the algebra $H$, which completes the claim.
\end{proof}

\begin{remark}
The skein $\cc$ has a multiplicative identity $1$, represented by the empty diagram.  Then $h\cdot 1$ is the element $h$ regarded as lying in the solid torus $V$.  In particular, when $h$ is a meridian element, meaning that $h$ is $P_{1,0}$ decorated by some element $c\in\cc$, the element $h\cdot 1$  is represented by the zero framed unknot in a ball inside $V$ decorated by $c$. Its value in $\cc$ is the scalar multiple  ${\rm ev}(c) 1$ of the identity, where ${\rm ev}(c)$ is the Homflypt polynomial of the unknot decorated by $c$.

Thus $P_{m,0}\cdot 1 ={\rm ev}(P_m) 1$, and it is known that $${\rm ev}(P_m) =\frac{v^{-m}-v^m}{s^m-s^{-m}}$$

Our commutation relation $[P_{m,0},P_{0,n}] =\{mn\}P_{m,n}$ leads to the calculation in \cite{MM08} of the effect in $\cc$ of putting a meridian decorated by $P_m$ around the core decorated by $P_n$. 
The core decorated by $P_n$ is $P_{0,n}\cdot 1\in\cc$.  Hence putting a meridian decorated by $P_m$ around this gives 
\beqn P_{m,0}\cdot(P_{0,n}\cdot 1)=P_{m,0}P_{0,n}\cdot 1 
&=&\{mn\}P_{m,n}\cdot 1+P_{0,n}P_{m,0}\cdot 1\\
& =&\{mn\}P_{m,n}\cdot 1 +{\rm ev}(P_m) P_{0,n}\cdot 1 \\
&=&\{mn\}P_{m,n}\cdot 1 +{\rm ev}(P_m) P_n.
\eeqn
In fact it was this equation in \cite[Thm.\ 18]{MM08},  illustrated there in figure 14, with $N, M$  in place of our $m,n$ respectively,  which encouraged us to conjecture (and then prove) the commutation relations of Theorem \ref{thm_commutationrelations} in their complete generality. 
\end{remark}

We now deduce formulae for the action of $P_{m,n}$ on the elements $Q_{\lambda,\mu}$. We first establish some notation.
\begin{definition}
We will use several statistics on pairs of partitions.
\begin{enumerate}
\item For $\lambda \subset \alpha$,  write $\alpha - \lambda$ for the skew partition consisting of the  cells contained in $\alpha$ but not in $\lambda$. 

\item Write $\lambda + n$ for the set of partitions $\alpha\supset\lambda$ where $\alpha - \lambda$ is a \emph{border strip}\footnote{A \emph{border strip} is a (skew) partition that is connected and contains no $2\times 2$ squares.} of length $n$. Similarly  $\lambda - n$ is the set of partitions $\beta\subset\lambda$ where $\lambda-\beta$ is a border strip of length $n$. 

\item Write $\mathrm{ht}(\gamma)$ for the \emph{height} of a border strip $\gamma$,  defined as the number of rows in $\gamma$ plus 1. 

\item Extend the  definition of content polynomial to cover skew partitions $\alpha-\lambda$ by setting \[C_{\alpha-\lambda}(t):=\sum_{x\in\alpha -\lambda} t^{c(x)}=C_\alpha(t)-C_\lambda(t).\]

Replace $s$ and $v$ by $s^m$ and $v^m$ in $ s_{\lambda,\mu}$ to define \[ s_{\lambda,\mu}(m):=\{m\}(v^{-m}C_\lambda(s^{2m})-v^mC_\mu(s^{-2m})) + \frac{v^{-m}-v^m}{s^m-s^{-m}}.\]

\item For a skew partition $\gamma$ write 
\beqn 
b(m,\gamma)&:=&v^{-m}C_\gamma(s^{2m})\\
b^-(m,\gamma)&:=&(-1)^{\mathrm{ht}(\gamma)} v^{-m}C_\gamma(s^{2m})
\eeqn
\end{enumerate}
\end{definition}

Using these definitions, we have the immediate relations that \[s_{\lambda,\mu}(m) =\{m\}b(m,\lambda)+\{-m\}b(-m,\mu)+\frac{v^{-m}-v^m}{s^m-s^{-m}}\] and $s_{\lambda,\mu}(-m)=s_{\mu,\lambda}(m)$.

\begin{remark}\label{strip} If $\gamma$ is a border strip of length $n$ it consists of a sequence of $n$ cells, each adjacent to the previous one, starting with a cell of least content $k$, say. The content increases by one for each cell in the sequence, so that $C_\gamma(t)=t^k(1+t+\cdots+t^{n-1})$. Then $b(1,\gamma)=v^{-1}s^{2k}(s^{2n}-1)/(s^2 -1)=v^{-1}s^{2k+n-1}\{n\}/\{1\}$. It follows that \[\frac{\{m\}}{\{mn\}}b(m,\gamma)=(v^{-1}s^{2k+n-1})^m.\]
\end{remark} 

 The following formulae generalise those of Theorem \ref{thm_Caction}. They apply even in the case of $n<0$, provided that when $\alpha\subset\lambda$ we set $\mathrm{ht}(\alpha-\lambda)=\mathrm{ht}(\lambda-\alpha)$, and still take $C_{\alpha-\lambda}(t)$ to mean $C_\alpha(t)-C_\lambda(t)$, so that $b(m,\alpha-\lambda)=-b(m,\lambda-\alpha)$.
\begin{theorem}\label{thm_Cactionfull} For $m,n \in \Z-\{0\}$ we have the following equalities:
 \begin{eqnarray}
  P_{m,0}\cdot Q_{\lambda,\mu} &=& s_{\lambda,\mu}(m)Q_{\lambda,\mu}\label{full1}\\
   P_{0,n}\cdot Q_{\lambda,\mu} &=& \sum_{\alpha \in \lambda + n} (-1)^{\mathrm{ht}(\alpha - \lambda)}Q_{\alpha,\mu} + \sum_{\beta \in \mu - n}(-1)^{\mathrm{ht}(\mu - \beta)}Q_{\lambda,\beta}\label{full2}\\
  P_{m,n}\cdot Q_{\lambda,\mu} &=& \frac{\{m\}}{\{mn\}} \left[
\sum_{\alpha \in \lambda + n}  b^-(m,\alpha-\lambda)Q_{\alpha,\mu} +
\sum_{\beta \in \mu - n}   b^-(-m,\mu - \beta) Q_{\lambda,\beta}\right] \label{full3}
\end{eqnarray}
 \end{theorem}

\begin{remark}
 Before proving the theorem, we remark that when $\mu = \emptyset$, these formulas were already known. The first follows from the identification of $P_{0,n}$ with the power sum function in \cite{Mor02, MM08}, and the second appears as \cite[Lemma 17]{MM08}.
\end{remark}

\begin{proof}
We know  the first two equations for $m = \pm 1$ and $n = \pm 1$. Equation (\ref{full3}) follows from (\ref{full1}) and (\ref{full2}) by a straightforward calculation using the commutation relation of Theorem \ref{thm_commutationrelations}. 

To prove equation (\ref{full1}), we proceed by induction on $m$, where in the inductive step we assume that the first and third equations are true for $0 \leq m \leq M$ and for $n \in \{-1,1\}$, and prove the first equation for $m = M+1$. (The case $m < 0$ follows by symmetry.) By the inductive assumption, we have
\begin{eqnarray*}
P_{m,1}\cdot Q_{\lambda,\mu} &=& \sum_{\alpha \in \lambda + 1} b(m,\alpha - \lambda)Q_{\alpha,\mu} + 
\sum_{\beta \in \mu - 1}b(-m,\mu - \beta)Q_{\lambda,\beta}\\
P_{1,-1} \cdot Q_{\lambda,\mu} &=& \sum_{\alpha \in \lambda - 1} b(1,\lambda - \alpha)Q_{\alpha,\mu} +
\sum_{\beta \in \mu + 1}b(-1,\beta - \mu)Q_{\lambda,\beta}
\end{eqnarray*}
From these, we can compute the action of $P_{m+1,0}$ using the commutation relation
\begin{equation}\label{eq_commtimesQ1}
P_{m+1,0}\cdot Q_{\lambda,\mu} = \frac{-1}{\{m+1\}\, } [P_{m,1},P_{1,-1}]\cdot Q_{\lambda,\mu}
\end{equation}
To shorten the following computation, we first note that in the formula for $P_{m,1}P_{1,-1}\cdot Q_{\lambda,\mu}$ there will be 4 types of terms corresponding to whether each operator adds cells to $\lambda$ or remove cells from $\mu$. The ``cross terms" where both 
$\lambda$ and $\mu$ change will cancel with the analogous cross terms from $-P_{1,-1}P_{m,1}\cdot Q_{\lambda,\mu}$. Also, equality (\ref{eq_commtimesQ1}) for the terms where just $\lambda$ changes is equivalent to the equality of the terms where just $\mu$ changes, by symmetry. Therefore, in the following computation we will just write the terms of the right hand side of (\ref{eq_commtimesQ1}) where only $\lambda$ changes, and denote the rest of the terms by ``$\cdots$''.
\begin{eqnarray}
P_{m,1}P_{1,-1}\cdot Q_{\lambda,\mu} &=& \sum_{\alpha \in \lambda - 1}b(1,\lambda - \alpha)\sum_{\alpha' \in \alpha + 1}b(m,\alpha' - \alpha)Q_{\alpha',\mu} + \cdots \label{eq_t1} \\
P_{1,-1}P_{m,1}\cdot Q_{\lambda,\mu} &=& \sum_{\beta \in \lambda + 1}b(m,\beta - \lambda)\sum_{\beta' \in \beta - 1}b(1,\beta - \beta')Q_{\beta',\mu}  + \cdots\label{eq_t2}
\end{eqnarray}
We now examine the coefficients of $Q_{\gamma,\mu}$ in $-[P_{m,1},P_{1,-1}]\cdot Q_{\lambda,\mu}$, which is the difference  of these two expressions. The terms where $\gamma \not= \lambda$ appear exactly once in both (\ref{eq_t1}) and (\ref{eq_t2}) with equal coefficients,  so they cancel. The coefficient of $Q_{\lambda,\mu}$  in (\ref{eq_t1}) comes from the cases where a cell $x$ is removed from $\lambda$ to give $\alpha\in\lambda-1$, and then restored to get $\alpha'=\lambda$. This gives \[\sum_{\{ x\,\mid \, \alpha = \lambda - x\} } b(1,x)b(m,x) = v^{-m-1}\sum_{\{ x \,\mid \,  \alpha = \lambda - x\} } s^{2(m+1)c(x)} .\] In (\ref{eq_t2}) we need $\beta'=\lambda$, and so $\beta$ arises by adding one cell $y$ to $\lambda$, to give the coefficient \[\sum_{\{ y \,\mid\,  \beta = \lambda + y\} } b(1,y)b(m,y)=v^{-m-1}\sum_{\{ y \, \mid \,  \beta = \lambda + y\} } s^{2(m+1)c(y)}.\]  
The  difference is then
\begin{eqnarray}
&{}& v^{-m-1}\sum_{\{ y \, \mid \,  \beta = \lambda + y\} } s^{2(m+1)c(y)} -v^{-m-1}\sum_{\{ x \,\mid \,  \alpha = \lambda - x\} } s^{2(m+1)c(x)} \label{eq_t3}
 \end{eqnarray}
 
It is now enough to show that the expression in (\ref{eq_t3}) is equal to the terms with coefficient $v^{-m-1}$ in  $\{m+1\}s_{\lambda,\mu}(m+1)$. (The terms with coefficient $v^{m+1}$ will come from the terms where a cell is added and subtracted from $\mu$, by symmetry). In other words, we must show that (\ref{eq_t3}) is equal to 
\begin{equation}\label{eq_t4}
v^{-m-1}  + \{m+1\}^2 b(m+1,\lambda) = v^{-m-1}\left[ 1+(s^{m+1} - s^{-m-1})^2 \sum_{z \in \lambda} s^{2(m+1)c(z)}\right]
\end{equation}
Finally, the equality of the expressions (\ref{eq_t3}) and (\ref{eq_t4}) is a well-known combinatorial identity. It can be proved by elementary means by first expanding the right hand side of (\ref{eq_t4}) along the rows and then along the columns of $\lambda$. The two powers of $\{m+1\}$ turn these two expansions into telescoping sums, and the leftover terms are exactly the ones in (\ref{eq_t3}). This completes the proof of  equation \ref{full1} in theorem \ref{thm_Cactionfull}.\\[2mm]

We now proceed to the proof of equation (\ref{full2}), using a similar induction on $n$. In this case we  use the commutation relation 
\begin{equation}\label{eq_commtimesQ2}
P_{0,n+1}\cdot Q_{\lambda,\mu} = \frac 1 {\{ n+1\}} [P_{1,n},P_{-1,1}]\cdot Q_{\lambda,\mu}
\end{equation}
The induction assumption shows that
\begin{eqnarray*}
\frac{\{n\} }{\{1\} } P_{1,n}\cdot Q_{\lambda,\mu} &=& \sum_{\alpha \in \lambda + n}b^-(1,\alpha - \lambda)Q_{\alpha,\mu} 
+ \sum_{\beta \in \mu - n}b^-(-1,\mu - \beta)
Q_{\lambda,\beta} \\
P_{-1,1}\cdot Q_{\lambda,\mu}&=& \sum_{\alpha \in \lambda + 1}b^-(-1,\alpha - \lambda)Q_{\alpha,\mu} 
+ \sum_{\beta \in \mu - 1}b^-(1,\mu - \beta)
Q_{\lambda,\beta}
\end{eqnarray*}
When $P_{1,n}P_{-1,1}$ is applied to $Q_{\lambda,\mu}$ there will be four types of terms, depending on  how cells are added to  $\lambda$ or subtracted from  $\mu$. 
Thus
\beqn 
\frac{\{n\} }{\{1\} }P_{1,n}P_{-1,1}\cdot Q_{\lambda,\mu}&=&\sum_{ \alpha\in\lambda+1, \gamma\in\alpha+n }a(\gamma,\mu)Q_{\gamma,\mu} + \sum_{ \alpha\in\lambda+1, \beta\in\mu-n }a(\alpha,\beta)Q_{\alpha,\beta}\\
&&+\sum_{  \beta\in\mu-1, \alpha\in\lambda+n}a'(\alpha,\beta)Q_{\alpha,\beta}+\sum_{\beta\in\mu-1, \gamma\in\beta-n }a(\lambda,\gamma)Q_{\lambda,\gamma}
\eeqn
There is a similar expansion
\beqn 
\frac{\{n\} }{\{1\} }P_{-1,1}P_{1,n}\cdot Q_{\lambda,\mu} &=&\sum_{ \alpha\in\lambda+n, \gamma\in\alpha+1 } d(\gamma,\mu)Q_{\gamma,\mu} + \sum_{ \beta\in\mu-n,\alpha\in\lambda+1  }d(\alpha,\beta)Q_{\alpha,\beta}\\
&& +\sum_{  \alpha\in\lambda+n,\beta\in\mu-1}d'(\alpha,\beta)Q_{\alpha,\beta}+\sum_{\beta\in\mu-n, \gamma\in\beta-1 } d(\lambda,\gamma)Q_{\lambda,\gamma}\eeqn where the operators are applied in the opposite order.

The coefficients of the  ``cross terms''  $Q_{\alpha,\beta}$ are 
 \beqn a(\alpha,\beta)&=& b(-1,\lambda-\alpha)b^-(1,\beta-\mu)\\
a'(\alpha,\beta)&=& b^-(1, \lambda-\alpha)b(-1,\beta-\mu)\\
d(\alpha,\beta)&=&b^-(1,\beta-\mu) b(-1,\lambda-\alpha)=a(\alpha,\beta)\\
d'(\alpha,\beta)&=& b(-1,\beta-\mu)b^-(1, \lambda-\alpha)=a'(\alpha,\beta)
\eeqn   
  These are unchanged when the operators are applied in the opposite order and thus
they will cancel to leave no cross terms on the right hand side of (\ref{eq_commtimesQ2}). 
  
  We now consider the coefficient $a(\gamma,\mu)-d(\gamma,\mu)$ of $Q_{\gamma,\mu}$ in $\frac{\{n\}}{\{1\}}[P_{1,n},P_{-1,1}]\cdot Q_{\lambda,\mu}$. We show that 
  \beqn
  a(\gamma,\mu)-d(\gamma,\mu)&=&\left\{ \begin{array}{ll} (-1)^{\mathrm{ht}(\gamma - \lambda)}\{n+1\}\frac{\{n\}}{\{1\} } & {\rm if } \ \gamma\in \lambda+(n+1)\\
   0 & {\rm  otherwise}
   \end{array}\right.
  \eeqn 
  
  Together with the similar result for the terms where only cells are removed from $\mu$ , relation (\ref{eq_commtimesQ2}) will establish  equation (\ref{full2}) for the coefficient of $P_{0,n+1}\cdot Q_{\lambda,\mu}$, and complete the induction step.

	The coefficients of $Q_{\gamma,\mu}$  are \beqn a(\gamma,\mu)&=&\sum_{(\alpha\in\lambda+1) \ \cap\ (\gamma-n)} b(-1,\alpha-\lambda)b^-(1,\gamma-\alpha)\\
	d(\gamma,\mu)&=&\sum_{(\alpha\in\lambda+n)\ \cap\ (\gamma-1)} b^-(1,\alpha-\lambda)b(-1,\gamma-\alpha)
	\eeqn 

We now consider the possible terms  in each sum, based on the shape of $\gamma - \lambda$. By construction $\gamma$ arises from $\lambda$, maybe in more than one way,  either by adding a cell   followed by an $n$-strip or  an $n$-strip followed by a cell. 

 If the cell $x$  and the $n$-strip $Y$ are disjoint then they can be added to $\lambda$ in either order. There will be only one term in each of $a(\gamma,\mu)$ and $d(\gamma,\mu)$,  both equal to $b(-1,x)b^-(1,Y)$, unless $n=1$ when $Y$ is also a cell, giving two terms in each sum.  In either case $a(\gamma,\mu)-d(\gamma,\mu)=0$.

Otherwise $\gamma-\lambda$ is connected and so it is either an $(n+1)$ border strip for $\lambda$ or it contains a single $2\times 2$ square.

 In the latter case the separately added cell $x$ must be one of the two cells on the south-west to north-east diagonal of the square, since they can't both occur in the $n$-strip $Y$. Now take $x$ to be the bottom left cell in the square and $x'$ the top right cell. Their respective complements $Y$ and $Y'$ in $\gamma-\lambda$ are $n$-strips. We can first add $x$ and then $Y$ or first $Y'$ and then $x'$ to get $\gamma$. Then   $a(\gamma,\mu)=b(-1,x)b^-(1,Y) =b(-1,x')b^-(1,Y')=d(\gamma, \mu)$, since the content of cells is constant on diagonals and $\mathrm{ht}(Y)=\mathrm{ht}(Y')$. Therefore only terms $Q_{\gamma,\mu}$ where $\gamma - \lambda$ is a border strip can have a non-zero coefficient.\\[2mm]

Consider now the coefficient $a(\gamma,\mu)-d(\gamma,\mu)$ when $\gamma - \lambda$ is a border strip. 
 There are two extreme cells in the border strip, $x$ at the top left, and $x'$ at the bottom right. Write $Y$ and $Y'$ respectively  for their complements in $\gamma-\lambda$. Write $h:=\mathrm{ht}(\gamma-\lambda)$ and $k:=c(x)$. Then $c(x')=k+n$, while the least content of a cell in $Y$ is $k+1$ and in $Y'$ is $k$.  By remark \ref{strip} we have $b^-(1,Y)=(-1)^{\mathrm{ht}(Y)}v^{-1}s^{2(k+1)+n-1}\{n\}/\{1\}$ and $b(-1,x)=vs^{-2k}$. In the same way  $b^-(1,Y')=(-1)^{\mathrm{ht}(Y')}v^{-1}s^{2k+n-1}\{n\}/\{1\}$ and $b(-1,x')=vs^{-2k-2n}$.
 
We can add $x$ and $Y$ to $\lambda$ in \emph{exactly} one order to get $\gamma$. If $x$ lies above $Y$ then we can add $x$ last, but not first. We have $\mathrm{ht}(Y) =h-1$, giving  a contribution of $b^-(1,Y)b(-1,x)=(-1)^{h-1}s^{n+1}\{n\}/\{1\}$ to $d(\gamma,\mu)$.    If $x$ lies to the left of $Y$ then we can add   $x$ first but not last. In this case $\mathrm{ht}(Y) =h$ and we get a contribution of $b^-(1,Y)b(-1,x)=(-1)^{h}s^{n+1}\{n\}/\{1\}$ to $a(\gamma,\mu)$.  In either case we get a contribution of  $(-1)^{h}s^{n+1}\{n\}/\{1\}$ to $a(\gamma,\mu)-d(\gamma,\mu)$.

In the same way we can add $x'$ and $Y'$ in exactly one order to get $\gamma$. When $x'$ lies below $Y'$ it can be added first. Then $\mathrm{ht}(Y')=h-1$ and we get the contribution  $b^-(1,Y')b(-1,x')=(-1)^{h-1}s^{-n-1}\{n\}/\{1\}$ to $a(\gamma,\mu)$. If $x'$ lies to the right of $Y'$ then it can be added last, but not first, and  we get the same contribution to $d(\gamma,\mu)$ with a changed sign, since $\mathrm{ht}(Y')=h$ in this case. Hence we have in either case a total of
 \beqn 
 a(\gamma,\mu)-d(\gamma,\mu)&=&(-1)^h (s^{n+1}-s^{-n-1}) \{n\}/\{1\}\\
&=& (-1)^{\mathrm{ht}(\gamma-\lambda)}\{n+1\}\frac{\{n\}}{\{1\}}
\eeqn
as claimed.  This completes the proof of  equation (\ref{full2}) and the proof of the theorem.
\end{proof}

\subsection{Further properties of $\cc$}

We collect here some results about $\cc$ as an algebra over $R$, and about its basis $Q_{\lambda,\mu}$.  The subalgebra $\cc^+$ spanned by $Q_{\lambda,\emptyset}$ is isomorphic to the ring $\Lambda$ of symmetric functions. The Schur function $s_\lambda$ corresponds to $Q_{\lambda,\emptyset}$. The identity element $1$ of $\cc$, represented by the  empty diagram, is given by $Q_{\emptyset,\emptyset}$. The element $P_n\in\cc$ lies in $\cc^+$ for $n>0$, and corresponds to the power sum $p_n\in\Lambda$. This interpretation of $\cc$ as symmetric functions was suggested in \cite{AM98}, with details established in \cite{Luk01,Luk05} and \cite{Mor02}.
 
\begin{remark}
 In the  case $\mu = \emptyset$, the formulae in theorem \ref{thm_Cactionfull} were already known.  Equation (\ref{full1}) appears in the context of meridian maps as \cite[Lemma 17]{MM08}.  
 
 A known result in the theory of symmetric   polynomials is  the expansion of the product of the $n$th power sum $p_n$ and the Schur function $ s_\lambda$ as a signed sum of Schur functions $s_\alpha$ where $\alpha-\lambda$ is an $n$-strip. Equation (\ref{full2}) then follows from the interpretation in \cite{MM08} of   $\cc^+$ as symmetric polynomials in which $P_n$ corresponds to $p_n$, and $Q_{\lambda,\emptyset}$ to    $s_\lambda$.  
\end{remark}

The subalgebra $\cc^+$ is spanned by closed braid diagrams in the annulus where all the strings go in the same direction.  Closed braids with strings in the reverse direction span an isomorphic subalgebra $\cc^-$, which is also spanned by $Q_{\emptyset,\mu}$.  Reversing string direction carries $Q_{\lambda,\mu}$ to $Q_{\mu,\lambda}$.
 Now $\cc$  can be written as $\cc\cong \cc^+\otimes\cc^-$, using Turaev's early description \cite{Tur88} of $\cc$ as a polynomial algebra. His generators $A_n$ are $P_{1,n}\cdot 1$ in the notation above, giving $\cc^+$ for $n>0$ and $\cc^-$ for $n<0$.
 
 We can then present the whole algebra $\cc$ as $\Lambda\otimes_R\Lambda$. We have already noted that  $P_k\in\cc^+$ represents $p_k$ when $k>0$ and hence $p_k\otimes 1$ in $\Lambda\otimes\Lambda$ while $P_{-k}$ for $k>0$ becomes $1\otimes p_{k}$. 
 
 The construction of the elements $Q_{\lambda,\mu}$ makes use of elements $h_n\in\cc^+$ corresponding to the complete symmetric functions in $\Lambda\otimes 1$ and $h_n^*\in\cc^-$, given by reversing the string direction, which become the complete symmetric functions in $1\otimes\Lambda$. 
 
 The formula in \cite{HM06} for $Q_{\lambda,\mu} $ is an extension of the classical Jacobi-Trudy formula for $s_\lambda$ as a polynomial in the complete symmetric functions. The general construction can be  illustrated by the case when $\lambda$ has parts ${2,2,1}$ and $\mu$ has parts ${3,2}$. Take a matrix with diagonal entries as shown, corresponding to the parts of $\lambda$ and $\mu$.

\[\begin{pmatrix} { h^*_2}&&&&\\
&{ h^*_3}&&&\\
&&{ h_2}&&\\
&&&{ h_2}&\\
&&&&{ h_1}
\end{pmatrix}\]

Complete the rows by shifting indices upwards for the parts of $\lambda$,
and downwards for the parts of $\mu$, to get
\[M=\begin{pmatrix} { h^*_2}&h^*_1&1&0&0\\
h^*_4&{ h^*_3}&h^*_2&h^*_1&1\\
1&h_1&{ h_2}&h_3&h_4\\
0&1&h_1&{ h_2}&h_3\\
0&0&0&1&{ h_1}
\end{pmatrix}\]

Then $Q_{\lambda,\mu}=\det M$.\\[2mm]

There is a further interesting interpretation for the whole of $\cc$, where we can consider $\Lambda\otimes\Lambda$ as symmetric functions in two sets of commuting variables $\xx$ and $\yy$, with the symmetric functions of $\xx$ representing the first copy of $\Lambda$ and the  symmetric functions of $\yy$ representing the second copy.

In this context there is a body of results stemming from work of King \cite{Kin70}, Koike \cite{Koi89} and subsequent authors in which such functions are studied, both as functions of two sets of variables, and in the special setting with $y_i=x_i^{-1}$ that deals with characters of $gl(N)$ for large $N$.  

Besides the Schur functions $s_\lambda(\xx)$ and $s_\lambda(\yy)$ King \cite{Kin70} defines `compound' Schur functions $s_{\lambda;\overline\mu}(\xx;\yy)$   by determinants that closely resemble those for $Q_{\lambda,\mu}$ in \cite{HM06}, or their counterpart in terms of elementary symmetric functions. As a result we can identify $Q_{\lambda,\mu}$ with the compound Schur function $s_ {\lambda:\overline\mu}$ of \cite{Kin70}. 

 Here are a few further facts about $\cc$ and its isomorphism with $\Lambda\otimes\Lambda$:
%
\begin{enumerate}
 \item The symmetry $a\otimes b \mapsto b\otimes a$ of $\cc$ sends $Q_{\lambda,\mu} \mapsto Q_{\mu,\lambda}$.
 \item The products $Q_{\alpha,\beta}Q_{\alpha',\beta'}$ expand as positive integer combinations in the basis $Q_{\lambda,\mu}$.
 \item The set $\{Q_{\lambda,\mu} \mid \lvert \lambda \rvert \leq n,\, \lvert \nu \rvert \leq p,\,\, \lvert \lambda \rvert - \lvert \mu\rvert = n - p\}$ spans  the subspace $\cc^{n,p}$ defined as the closure of $(n,p)$ diagrams in the square.
  \item $\cc = \oplus \cc^{n,p}$, where $\cc^{n,p} \subset \cc^{n+1,p+1}$ and $\cc^{n,p} \cap \cc^{n',p'} = 0$ if $n-p \not= n'-p'$.
 \item  $Q_{\lambda,\emptyset} = s_\lambda \otimes 1$, and $Q_{\emptyset,\mu} = 1\otimes s_\mu$.
 \item We have $Q_{\lambda,\mu} = Q_{\lambda,\emptyset}Q_{\emptyset,\mu} + v = (s_\lambda \otimes 1)(1\otimes s_\mu) + x$, for some $x \in \cc^{\lvert \lambda \rvert - 1,\lvert \mu \rvert - 1}$.
\end{enumerate}

\begin{remark}
 Fact 5 is established by Lukac, \cite[Ch.\ 3]{Luk01} or \cite{Luk05}. The other facts appear in \cite{HM06} and \cite{MH02}.  Fact 1 is immediate from  the determinantal formula  on reversing the orientations of all curves and rotating the matrix. Fact 2 is theorem 3.5 in \cite{HM06}, while facts 3,4 and 6 are in \cite{MH02}. Fact 6, along with a more detailed expression for $x$,  can also be deduced from \cite{Kin70} (see also \cite{Koi89}), when $Q_{\lambda,\mu}$ is interpreted in terms of compound Schur functions. \end{remark}

\section{The elliptic Hall algebra}\label{sec_ehall}
In this section we recall from \cite{BS12} a presentation of the elliptic Hall algebra $\E_{q,t}$, which depends on two parameters $q,t \in \C^*$. For the convenience of the next section, we will switch $t \mapsto t^{-1}$ from the notation of \cite{BS12}. We then prove that the $t=q$ specialization $\E_{q,q}$ is isomorphic to the Homflypt algebra $H_{s=q^{-1/2},v}$.

\begin{remark}
Before giving a presentation, we recall a short description of the construction of the algebra $\E_{q,t}$ from the introduction of \cite{BS12}. First, we consider a smooth elliptic curve $X$ over $\mathbb F_p$, and the category $\mathrm{Coh}(X)$ of coherent sheaves over $X$. The \emph{Hall algebra} of this category is a (topological) bialgebra $\E^+_{\sigma,\bar \sigma}$, where $\sigma, \bar \sigma$ are the Frobenius eigenvalues on the $l$-adic cohomology group $H^1(X_{\bar{\mathbb{F}}_p},\overline{\mathbb{Q}_l})$. It is proved in \cite{BS12} that the relations can be written entirely in terms of Laurent polynomials in these parameters, so we rename the parameters $q,t$ and allow them to be formal (i.e. $\E^+_{q,t}$ is an algebra over $\C[q^{\pm 1},t^{\pm 1}]$). Then $\E_{q,t}$ is the Drinfeld double of the algebra $\E^+_{q,t}$.
\end{remark}

As before, we will write $d(\xx) = gcd(a,b)$ if $\xx = (a,b)$, and $d(\xx,\yy) = \det[\xx\,\yy]$ for $\xx,\yy \in \Z^2$. Define the constant
\[
\alpha_i := (1-q^{i})(1-t^{-i})(1-q^{-i}t^{i}) / i
\]
\begin{definition}
By \cite[Thm.\ 5.4]{BS12}, the \emph{elliptic Hall algebra} $\E_{q,t}$ is generated by elements $u_\xx$ for $\xx \in \Z^2$, with relations
\begin{enumerate}
 \item If $\xx,\yy$ belong to the same line in $\Z^2$, then
 \[
  [u_\xx,u_\yy] = 0
 \]
\item If $\xx,\yy \in \Z^2$ are such that $d(\xx) = 1$ and $\Delta_{\xx,\yy}$ has no interior lattice points, then
\begin{equation}\label{eq_hallrelation1}
 [u_\yy,u_\xx] = \epsilon(\xx,\yy) \frac{\theta_{\xx+\yy}}{\alpha_1}
\end{equation}
\end{enumerate}
where $\epsilon(\xx,\yy) := \mathrm{sign}(d(\xx,\yy))$ and the elements $\theta_\xx$ are polynomials in the $u_{k\xx}$ defined for $d(\xx_0)= 1$ by equating the following series:
\begin{equation}\label{eq_hallrelation2}
 1 + \sum_{i > 0}\theta_{i \xx_0}z^i = \mathrm{exp}\Big(\sum_{r \geq 1}\alpha_r u_{r \xx_0}z^r\Big)
\end{equation}
where $z$ is a formal variable.
\end{definition}

\begin{remark}\label{rmk_sl2actiononEhall}
By \cite[Lemma 5.3]{BS12}, the group $\SL_2(\Z)$ acts on $\E_{q,t}$ via $\gamma(u_\xx) = u_{\gamma(\xx)}$.
\end{remark}

These relations look somewhat similar to the commutation relations for $H$, but they are complicated by the definition of $\theta_\xx$. The key observation is that if $t=q$, then the element $\theta_\xx$ defined in (\ref{eq_hallrelation2}) simplifies substantially.
\begin{lemma}
 If $t=q$, then 
 \[
 \frac{\theta_\xx}{\alpha_1} = \left([d(\xx)]_{q^{1/2}}\right)^2 u_\xx
 \]
\end{lemma}
\begin{proof}
Each constant $\alpha_i$ has a zero of order 1 at $t=q$. If we write the RHS of (\ref{eq_hallrelation2}) as $\mathrm{exp}(a)$, then only the terms of degree 1 in $a$ have a simple zero at $t=q$. Therefore, if we specialize $t=q$, the only surviving term in $\mathrm{exp}(a) / \alpha_1$ is $a/\alpha_1$, so the identity $\left(\alpha_i / \alpha_1\right)|_{t=q} = \left( [i]_{q^{1/2}}\right)^2$ shows the claim.
\end{proof}

\begin{corollary}\label{lemma_relatteq1}
 If $t=q$, then the following relations are satisfied:
\begin{align}
  \,[u_{1,0}, u_{-1,k}] &= - \mathrm{sign}(k) \left([k]_{q^{1/2}}\right)^2 u_{0,k}\notag\\
  \,[u_{1,0}, u_{0,k}] &=  - \mathrm{sign}(k)  u_{1,k}\notag
 \end{align}
\end{corollary}
We now define renormalized generators
$
w_\xx :=\left(q^{d(\xx)/2} - q^{-d(\xx)/2}\right) u_\xx
$.
\begin{theorem}\label{thm_HisotoE}
 If we specialize $q=t$ and identify $t = q = s^{-2}$, then the map $P_\xx \mapsto w_\xx$ extends to an $\SL_2(\Z)$-equivariant $\Z^2$-graded isomorphism of algebras $H_{s,v} \to \E_{q=s^{-2},t=s^{-2}}$.
\end{theorem}
\begin{proof} 
We first remark that $\SL_2(\Z)$ acts by permutation on the generators $w_\xx$ by Remark \ref{rmk_sl2actiononEhall} (since it preserves the $gcd$ of the entries of vectors). If we rewrite the relations of Corollary \ref{lemma_relatteq1} in terms of $w_\xx$ and the parameter $s$, we obtain
 \begin{align}
  \,[w_{1,0}, w_{-1,k}] &= \{k\}_s w_{0,k}\notag\\
  \,[w_{1,0}, w_{0,k}] &=  \{k\}_s w_{1,k}\notag
 \end{align}
 These are the same as the relations (\ref{formula_somerelations}), so Remark \ref{rmk_sl2actiononEhall} combined with Proposition \ref{prop_allfromsome} shows that the $w_\xx$ also satisfy the relations (\ref{formula_allrelations}). The map is clearly surjective, and it is injective because the description of the basis of $H$ in Corollary \ref{cor_basis} agrees with the PBW basis of $\E_{q,t}$ described in \cite[Thm.\ 4.8]{BS12}.
\end{proof}

\begin{remark}
 There is $S_3$ symmetry in the parameters $\{q,t,qt^{-1}\} \in (\C^*)^3$, so if we specialize $s^2=q^{-1}$ and $t=1$ the previous theorem remains true. 
\end{remark}

\section{Adaptations of the Homflypt skein relations}\label{sec_adaptations}

There are a number of instances, for example in the context of families of Hecke algebras of type $A$, or in relation to quantum $\SL_N$ modules and associated invariants, where  Homflypt skeins can be used as models after a simple adaptation.

The simplest model of the Hecke algebra $H_n$ of type $A_{n-1}$ is   the Homflypt skein of oriented framed $n$-tangles, using diagrams in a rectangle with $n$ inputs at the bottom and $n$ outputs at the top \cite{MT90}. Composition is induced by stacking diagrams and the algebra is generated by the elementary $n$-braids \bc $\sigma_i$\quad=\quad \labellist\small
\pinlabel {$i$} at 200 405
\pinlabel {$i+1$} at 250 405
\endlabellist \sigmaior.\\[6mm]\ec

Write $T_i$ for the element of the skein represented by $\sigma_i$. The basic skein relation  gives the equation $T_i -T_i^{-1}=(s-s^{-1}) {\rm Id}$, and hence the quadratic relation \[(T_i -s)(T_i +s^{-1})=0,\] with roots $s, -s^{-1}$.

Many algebraic accounts use a version of the Hecke algebra where the quadratic has roots $q, -1$, so it is useful to adapt the skein theory to allow for roots $xs, -xs^{-1}$ with an extra parameter $x$. This is done by Aiston and Morton in \cite{AM98} for the Hecke algebra, and subsequently used in the form below for other skeins.

\subsection{The adaptable Homflypt skein}

Use $R[x^{\pm1}]$-linear combinations of framed oriented curves in a $3$-manifold $M$, possibly including arcs with fixed input and output points on $\partial M$, subject to the relations
\bc
$x^{-1}\Xor - x\Yor=(s-s^{-1})\ \Ior$ \qquad (Switch and smooth)\\[2mm]

$\Rcurlor=xv^{-1}\ \Idor\ ,\qquad \Lcurlor =x^{-1}v\ \Idor$ \qquad (Framing change)\\[2mm]
\ec
with the local blackboard framing convention. The resulting skein $H_x(M)$ provides the relations $x^{-1}T_i -xT_i^{-1}=(s-s^{-1}){\rm Id}$ and hence the quadratic relation with roots $xs, -xs^{-1}$.

This is useful in several instances.
\begin{itemize}
\item
Take $x=s$ and set $q=s^2$ to recover the algebraic version of the Hecke algebra with roots $q, -1$.

\item
Take $x=v$ to eliminate the framing dependence.
\item
Take $s=e^{h/2}, v=s^{-N}, x=e^{-h/2N}=s^{-1/N}$ to adjust for the quadratic relation satisfied by the fundamental $R$-matrix of the $SL_N$ quantum group, and the effect of framing change when constructing knot invariants. \cite{Ais96, MM08}
\end{itemize}

Much of Aiston's original work uses these adaptable relations, with $x$ as an indeterminate alongside $v$ and $s$ in the coefficient ring.

Clearly, knowing the skein $H_x(M)$ we can find the basic skein $H(M)=H_1(M)$ by setting $x=1$.  
 Lukac suggested how to reverse the process in many instances and recover $H_x(M)$ from $H(M)$, so that we can work  without $x$, while still being in a position to adapt if needed.

\begin{theorem}  When $M=F\times I$ is a thickened surface there is a linear isomorphism $f_x:H(M)\to H_x(M)$.
\end{theorem}
\begin{proof} Represent each union of framed curves in $M$ by a diagram $D$ on $F$ with the blackboard framing. The allowed changes in the curves  alter $D$ by isotopy in $F$ and Reidemeister moves $R_{II}, R_{III}$. The \emph{writhe} of $D$, $w(D)$, defined as the sum of the signs of the crossings in $D$,  then depends only on the curves in $M$ and not on the choice of representing diagram.

Define $f_x$ on diagrams by \[f_x(D)=x^{-w(D)} D.\]

To prove that this induces a well-defined map on $H(M)$ we must show that the skein relations are respected.

For the switch and smooth relation we must show that \[f_x(D_+)- f_x(D_-)=(s-s^{-1})f_x(D_0)\] in $H_x(M)$, 
where three diagrams $D_+, D_-, D_0$ differ only by switching or smoothing a crossing. 

Now the writhes of $D_{\pm}, D_0$ satisfy $w(D_+)=w+1, w(D_-)=w-1$ where $w=w(D_0)$, so 
\[f_x(D_+)- f_x(D_-)=x^{-w-1}D_+-x^{-w+1}D_-=x^{-w}(s-s^{-1})D_0 =(s-s^{-1})f_x(D_0)\] in $H_x(M)$.

Similarly, for the framing change, $w\left(\Rcurlor\right)=w+1$ where $w=w\left(\Idor\right)$, so in $H_x(M)$ we have
\[f_x\left(\Rcurlor\right)=x^{-w-1}\  \Rcurlor=x^{-w}v^{-1}\ \Idor=v^{-1}f_x\left(\Idor\right).\]
\end{proof}

For example, if we need to adapt the element $P_m=(s-s^{-1})/(s^m-s^{-m}) X_m$ from our algebra $H$ above to $H_x$ we replace $X_m=\sum A_{i,j}$ by $\sum x^{j-i}A_{i,j}$ as in Aiston's original version. The product $P_\xx P_\yy$ in $H$ is replaced by $x^{-k}P_\xx P_\yy$ on passing to $H_x$ where we use the adapted $P_\xx,P_\yy$ in $H_x$, and set $k=\det[\xx\  \yy]$. This implies the following corollary of Theorem \ref{thm_commutationrelations}:
\begin{corollary}
Using the general Homflypt skein relations in this section (with parameters $x,s,v$), the algebra $H$ is generated by elements $P_\xx$ with relations
\[ 
x^{-k}P_\xx P_\yy- x^kP_\yy P_\xx = \{d\}P_{\xx,\yy} 
\]
where $d = \det [\xx\,\yy]$.
\end{corollary}

The quadratic relations used by Schiffman, Vasserot, and Cherednik correspond to the basic skein, so there is largely no need for adaptation.  However, comparison with the results  of Frohman and Gelca \cite{FG00} for the Kauffman bracket skein  needs the adaptation, after orienting, of $x=-A^{-1}$, $s=A^{-2}$ and $v=A^{-4}$.

\section{Iterated Cables}\label{sec_iteratedcables}
Let $K$ be an iterated cable of the unknot and $\lambda$ a partition. In this section we use the isomorphism between the elliptic Hall algebra and the Homflypt skein algebra to construct a 3-variable polynomial that specializes to the $\lambda$-colored Homflypt polynomial of $K$ (up to a monomial $s^\bullet v^\bullet$). This can be considered to be an $\sl_\infty$ version of the construction in \cite{Sam14} for $\g = \sl_2$, which was generalized in \cite{CD14} to arbitrary $\g$. Our construction uses the work of Schiffmann and Vasserot \cite{SV13, SV11} in an essential way, and when restricted to torus knots, it is essentially the same as the construction of Gorsky and Negut \cite{GN13}. We will use our construction to prove a conjecture of Cherednik and Danilenko in \cite{CD14}.

We first establish some notation. Throughout the section, $\mm = (m_1,\ldots,m_k)$ and $\nn = (n_1,\ldots,n_k)$ will be sequences of integers with $m_i, n_i$ relatively prime and $m_i>0$. We will write $\Lambda^N := \C[x_1,\cdots,x_N]^{S_N}$ for the graded ring of symmetric polynomials, and $\Lambda$ for the (graded) ring of symmetric functions, which is $N \to \infty$ limit of the $\Lambda^N$. We write $\pi_N: \Lambda \to \Lambda^{N}$ for the natural projection.

There are three algebras that we will use in this section: the Homflypt skein algebra $H$, the elliptic Hall algebra $\E_{q,t}$, and the double affine Hecke algebra $\H^N_{q,t}$ (defined below). In general we will use superscripts $H$, $\E$, and $N$ to distinguish between objects associated to these three algebras. For example, associated to the sequences $\mm,\nn$ and a partition $\lambda$ we will define two polynomials using the representation theory of $\E_{q,t}$ and $\H^N_{q,t}$, respectively:
\begin{equation*}
J^\E(\mm,\nn,\lambda; q,t,u) \in \C[q^\pmone, t^\pmone,u^\pmone],\quad J^N(\mm,\nn,\lambda; q,t) \in \C[q^\pmone, t^\pmone]
\end{equation*}

(Technically, we actually define rational functions - see Remark \ref{rmk_rational}.) We will relate these polynomials to the colored Homflypt polynomial of the iterated cable $K(\mm,\nn)$ of the unknot determined by the sequences $\mm$ and $\nn$. We first define the notion of iterated cable that we will use:

\begin{definition}\label{def_ourcable}
Let $K$ be a framed knot, let $T$ be the torus which bounds a neighborhood of $K$, and let $L_{fr}$ be the longitude in $T$ determined by the framing of $K$. 
\begin{enumerate}
 \item The \emph{algebraic} $(m,n)$ cable of $K$ is the framed knot in $T$ such that
\[ 
K(m,n) \sim m L_{fr} + n M
\]
(In this notation, the symbol $\sim$ means `is homologous to' and $M$ is the meridian of $K$.) The framing of $K(m,n)$ is defined to be parallel to the torus $T$.  
\item We then define a framed knot $K(\mm,\nn)$ inductively as follows: $K(m_1,n_1)$ is the algebraic $(m_1,n_1)$ cable of the 0-framed unknot, and $K(\mm_k,\nn_k)$ is the algebraic $(m_k,n_k)$ cable of $K(\mm_{k-1},\nn_{k-1})$. 
\end{enumerate}
\end{definition}

\begin{remark}
 We note that the algebraic cabling procedure is not the standard topological construction of a cable of a knot. However, the resulting cabling formula (Prop. \ref{prop_cabling}) is particularly simple, which makes it convenient for our purposes. This cabling procedure is related to algebraic knots, which are the knots obtained by intersecting a (singular) irreducible algebraic curve in $\C^2$ with a small copy of $S^3$ around the singularity. In particular, if the $\mm$ and $\nn$ are the \emph{Newton pairs} of an algebraic plane curve (see \cite[Appendix to Ch.\ 1]{EN85}), then the algebraic knot obtained from this curve is $K(\mm,\nn)$. The knots that arise this way are exactly those with $n_i > 0$. (These knots are determined by their Alexander polynomial.)
 
 We will not need to discuss algebraic knots, but it is worth calling attention to \cite[Conj.\ 2.4(iii)]{CD14}, which states that if the $\mm$, $\nn$ are the Newton pairs of an algebraic knot, then the specialization $J^\E(\square;q=1,t,u=0)$ is related to the Betti numbers of the Jacobian factor of the curve. It is not clear if the skein-theoretic point of view in this paper can say anything about this conjecture.
\end{remark}

We will use the following as our definition of the Homflypt polynomial of $K(\mm,\nn)$. This definition differs from the standard definition of the Homflypt polynomial of $K$ by a monomial $s^\bullet v^\bullet$ depending on $\mm$, $\nn$ and $\lvert \lambda \rvert$, but we will ignore this difference since the conjecture in \cite{CD14} is stated up to an overall constant.

\begin{definition}\label{def_ourhomfly}
The evaluation in the skein of $S^3$ of the framed knot $K(\mm,\nn)$ colored by the element $Q_\lambda \in \cc^+$ will be denoted as follows:
\begin{equation*}
J^H(\mm,\nn,\lambda; v,s) \in Homflypt(S^3) = \C[v^\pmone, s^\pmone,(s^k - s^{-k})^{-1}]
\end{equation*}
\end{definition}

In this section we will prove the following theorem:
\begin{theorem}\label{thm_iteratedcable}
For $\mm,\nn,\lambda$ as above, we have the following specializations:
\begin{eqnarray}
v^\bullet s^\bullet J^\E(\mm,\nn,\lambda; q,t,u)\Big|_{q=s^{-2}, t=s^{-2}, u = v^2} &=&  J^H(\mm,\nn,\lambda; v,s) \label{eq_equality1}
\\
{ } u^\bullet J^\E (\mm,\nn,\lambda; q,t,u)\Big|_{u = t^N} &=& q^\bullet t^\bullet  J^N(\mm,\nn,\lambda; q,t)\label{eq_equality2}
\end{eqnarray}
(where the powers denoted by ``$\bullet$'' depend on $\mm$,$\nn$, and $\lvert \lambda\rvert $, but not on $N$). In particular, the Connection Conjecture \cite[Conj.\ 2.4(i)]{CD14} is true.
\end{theorem}

\begin{remark}
The existence of a polynomial $J^\E$ satisfying the second specialization was announced as a theorem in \cite{CD14}. The proof of this is essentially identical to the proof in \cite{GN13} (which used the results in \cite{SV11} and  \cite{SV13}). For the sake of completeness we will include this proof in Section \ref{sec_dahapoly}. We also remark that the stabilization variable $a$ in the Connection Conjecture of \cite{CD14} is $-u$ for us, so their specialization $a = -t^N$ becomes our specialization $u = t^N$.
\end{remark}

\subsection{Homflypt cabling formula}
In this section we give a cabling formula for the Homflypt polynomial $J^H(\mm,\nn,\lambda; v,s)$ of the $\lambda$-colored framed knot $K(\mm,\nn)$. In particular, this will give a algebraic formula for $J^H$ in terms of the action of $H$ on the skein $\cc$ of the annulus. This formula will later be compared to a specialization of the formula given in Section \ref{sec_epoly} which defines $J^\E$, and this will imply the first equality of Theorem \ref{thm_iteratedcable}. 

To simplify comparison to the elliptic Hall algebra, we will need to twist the action of $H$ on $\cc$ by an automorphism. To try to make this section self-contained, we will recall the necessary facts about $H$ and $\cc$ before giving the cabling formula.

\subsubsection{Notation}
We will twist the action of $H$ on $\cc$ by the automorphism $P_{m,n} \mapsto P_{-n,m}$. To compare with the constructions in the following sections we will use the following definitions.
\begin{definition}
We define the following subalgebras of $H$:
\[
H^\geq := \langle P_{m,n} \mid m \geq 0\rangle,\quad\quad H^> := \langle P_{m,n} \mid m > 0 \rangle
\]
We also will use the following $R$-submodule of $\cc$
\[
\cc^+ := R\{Q_{\lambda,\emptyset}\} \stackrel{\sim}\to \Lambda,\quad\quad Q_{\lambda,\emptyset} \mapsto s_\lambda
\]
\end{definition}
The map $\cc^+ \stackrel \sim \to \Lambda$ is an \emph{algebra} isomorphism by \cite[Thm.\ 8.2]{Luk05}. The action of $H^\geq$ preserves the subspace $\cc^+ \subset \cc$, so we can identify $\Lambda$ as an $H^\geq$-module. This module structure is described as follows.

\begin{lemma}\label{lemma_actiononLambda}
 The action of $H$ on $\Lambda$ is given by
\begin{eqnarray*}
P_{m,0}\cdot s_\lambda &=& p_m s_\lambda\\
P_{0,n}\cdot s_\lambda &=& \left[\frac{v^{-n} - v^n}{s^n - s^{-n}} + v^n(s^{-n} - s^n)\sum_{x \in \lambda} s^{-2c(x)}\right] s_\lambda \\
&=& \left[\frac{v^{-n} - v^n}{s^n - s^{-n}} + v^n s^{-n} \sum_{i=1}^k (s^{-2n\lambda_i} - 1)s^{2ni}\right]s_\lambda
\end{eqnarray*}
where $\lambda = (\lambda_1,\ldots,\lambda_k)$.
\end{lemma}
\begin{proof}
 The first equality follows from the fact that when $\mu = \emptyset$, the second equation of Theorem \ref{thm_Cactionfull} agrees with the Murnaghan-Nakayama rule for the product $p_m s_\lambda$ of a power sum times a Schur function. The second equality is translated from Theorem \ref{thm_Cactionfull}, and the third equality follows from the fact that the sum for each row is a telescoping sum.
\end{proof}

\begin{remark}\label{rmk_previouslyproved}
 These equations have been proved previously by Morton and coauthors. The first follows from the identification of $P_{m,0}$ with the power sum function in \cite{Mor02, MM08}, and the second appears as \cite[Lemma 17]{MM08}.
\end{remark}

If $K$ is a framed knot, then the inclusion $N_K \hookrightarrow S^3$ induces an $R$-linear evaluation map $\ev^H_K: \cc \to R$. For later use we recall an explicit formula for the evaluation map $\ev_U^H$ restricted to the subspace $\Lambda\subset \cc$ when $U$ is the $0$-framed unknot. 
\begin{lemma}[{\cite[eq.\ (12)]{ML03}}]
For the $0$-framed unknot $U$, the evaluation map is
\[
\ev^H_U: \Lambda \to R,\quad \quad \ev^H_U(s_\lambda) = \prod_{x \in \lambda} \frac{v^{-1}s^{c(x)} - vs^{-c(x)}}{s^{hl(x)} - s^{-hl(x)}}
\]
\end{lemma}
(If $x \in \lambda$ is a box, then $c(x)$ and $hl(x)$ are its content and hook length, see Section \ref{sec_notation}.)

\begin{remark}
It was shown in \cite[Thm.\ 1]{Mor07} that the evaluation $\ev^H_K(Q_{\lambda,\mu})$ is divisible by the evaluation $\ev_U^H(Q_{\lambda,\mu})$. (The latter evaluation is typically called the quantum dimension.)
\end{remark}

\subsubsection{The cabling formula}
Let $K$ be a framed knot with $N_K$ a tubular neighborhood of $K$ and $T_K$ the boundary of $N_K$. We use the framing of $K$ to identify $T$ with the standard torus, and we use this to identify $H$ with the skein of $T_K$ and $\cc$ with the skein of $N_K$. Under these identifications, the framed knot $K$ is isotopic to $P_{1,0}\in H$ and is also isotopic to $Q_{(1),\emptyset} \in \cc^+ \subset \cc$. 

Let $K(m,n)$ be the algebraic $(m,n)$ cable of $K$ and let $N_{m,n}$ be a tubular neighborhood of $K(m,n)$. We identify $\cc$ with the skein of $N_{m,n}$ using the framing of $K(m,n)$ - in particular, under this identification, the element $Q_{(1),\emptyset}$ is isotopic to $K(m,n)$.

Finally, let $\Gamma_{m,n}: N_{m,n} \to N_K$ be the inclusion, let $\iota_{m,n}^H: \cc \to H$ be the inclusion given by inserting the annulus along the $(m,n)$ curve. We choose 
\begin{equation}\label{eq_gammamn}
\gamma_{m,n} \in \SL_2(\Z)\textrm{ such that } \gamma_{m,n}\left(\begin{array}{c}1 \\ 0 \end{array}\right) = \left(\begin{array}{c} m \\ n \end{array}\right)
\end{equation}
We will write 
\[ \Gamma^H_{m,n}: \cc \to \cc
\] for the $R$-linear map induced by $\Gamma_{m,n}$ (where we have identified $H(N_{m,n})$ and $H(N_K)$ with $\cc$ as described above). Then the following lemma follows immediately from our choices of identification. (See, e.g. \cite[Lemma 2.20]{Sam14}.)
\begin{lemma}\label{lemma_cablingmap}
Under the identifications above, the $R$-linear map $\Gamma^H_{m,n}: \cc \to \cc$ is 
\[
\Gamma^H_{m,n}(x) = \gamma_{m,n}(\iota_{1,0}(x))\cdot 1 = \iota_{m,n}(x)\cdot 1
\]
\end{lemma}

Given sequences $\mm,\nn$ as before, we will use the composition
\[
\Gamma^H_{\mm,\nn} := \Gamma_{m_1,n_1}^H \circ \cdots \circ \Gamma_{m_k, n_k}^H
\]

\begin{proposition}\label{prop_cabling}
If $U$ is the $0$-framed unknot we have the following equality:
\[
 J^H(K(\mm,\nn),\lambda;s,v) = \ev^H_{K(\mm,\nn)}(Q_\lambda) = \left( \ev^H_U \circ \Gamma^H_{\mm,\nn}\right) (Q_\lambda)
\]
In particular, the $\lambda$-colored Homflypt polynomial $J^H(\mm,\nn,\lambda; v,s)$ of the iterated cable $K(\mm,\nn)$ (in our normalization, see Definition \ref{def_ourhomfly}) is equal to the right hand side.
\end{proposition}
\begin{proof}
This follows from Lemma \ref{lemma_cablingmap} and our choices of identification. First, the left hand side of the equation is induced from the inclusion of $N_{K(\mm,\nn)}\hookrightarrow S^3$, and this inclusion defines the ($\lambda$-colored) Homflypt polynomial. Then the right hand side is induced from the sequence of inclusions 
\[ 
N_{K(\mm_k,\nn_k)} \hookrightarrow N_{K(\mm_{k-1},\nn_{k-1})} \hookrightarrow \cdots \hookrightarrow N_{K(\mm_1,\nn_1)} \hookrightarrow N_U \hookrightarrow S^3
\]
and the composition of these inclusions is equal to the inclusion $N_{K(\mm,\nn)} \hookrightarrow S^3$.
\end{proof}

\begin{remark}
 This proposition shows that the colored Homflypt polynomials of an iterated cable of the unknot can be evaluated using the skein algebra $H$, the $\SL_2(\Z)$ action on $H$, the action of $H$ on $\Lambda$, and the evaluation map on $\Lambda$. The more standard way of writing the cabling formula gives an expression for the polynomials of the cable of $K$ in terms of the polynomials of $K$. Such an expression could be derived from our cabling formula, but we will not need to do this. However, in the simpler case of the Kauffman bracket skein module, both versions of the cabling formula have appeared in multiple places. (Precise statements appear in \cite[Cor.\ 2.15, Cor.\ 2.16]{Sam14}, together with references to other versions.)
 
 In \cite[Prop.\ 4.2]{CD14}, the $\sl_2$ version of the latter cabling formula was used to prove the analogue of the specialization in equation (\ref{eq_equality1}), and the same specialization was proved in \cite[Thm.\ A.8]{Sam14} using the former cabling formula (again for $\sl_2$). It therefore should not be surprising that our version of the Homflypt cabling formula implies the specialization in equation (\ref{eq_equality1}), after the Homflypt skein algebra and elliptic Hall algebra have been shown to be isomorphic.
\end{remark}

\subsection{The Hall algebra}\label{sec_epoly}
In this section we define a 3-variable polynomial that specializes to the Homflypt polynomial $\ev^H_{K(\mm,\nn)}(s_\lambda)$ of Proposition \ref{prop_cabling}. The key idea is that all the objects on the right hand side of the equality in Proposition \ref{prop_cabling} have $t$-deformations that come from the elliptic Hall algebra $\E_{q,t}$. We will use the work of Schiffmann and Vasserot in \cite{SV11,SV13}, where they constructed an action of a subalgebra of $\E_{q,t}$ on $\Lambda$. However, to simplify comparison to double affine Hecke algebras we will change $t \mapsto t^{-1}$ (as in Section \ref{sec_ehall}). We also use their renormalized generators $v_\xx := (q^{d(x)} - 1)u_\xx$. 

We first define subalgebras
\[
\E^\geq_{q,t} := \langle v_{m,n} \mid m \geq 0\rangle,\quad \quad \E^>_{q,t} := \langle v_{m,n} \mid m > 0\rangle
\]
We now recall from \cite[Prop.\ 1.4]{SV13} an action of $\E_{q,t}^\geq$ on $\Lambda$. The action will be written in terms of \emph{Macdonald polynomials} $P_\lambda \in \Lambda$, which form a basis for $\Lambda$. The element $v_{1,0}$ acts by multiplication by the power sum $p_1$, and for $k \geq 1$ we have
\begin{eqnarray}\label{eq_Eaction}
v_{0,k}\cdot P_\lambda &:=& \left( \sum_i (q^{k \lambda_i} - 1)t^{-k(i-1)}\right) P_\lambda \\
v_{0,-k}\cdot P_\lambda &:=& q^k\left( \sum_i (q^{-k \lambda_i} - 1)t^{k(i-1)}\right) P_\lambda\notag
\end{eqnarray}

\begin{remark}\label{rmk_homogeneous}
 The Macdonald polynomials $P_\lambda$ are a homogeneous basis of $\Lambda$. Then equation (\ref{eq_Eaction}) combined with the presentation of $\E_{q,t}$ implies that the operators $v_{m,n}$ are graded operators, in the sense that they take homogeneous elements to homogeneous elements. In particular, the action of $\E^+_{q,t}$ on $\Lambda$ is graded, which generalizes property (4) under Definition \ref{def_homflyskeinmod}.
\end{remark}

In fact, the subalgebra $\E_{q,t}^{hor} := \langle v_{k,0} \mid k \geq 0 \rangle$ is isomorphic to $\Lambda$ as a graded algebra, and it acts by multiplication operators. Given $m,n \in \Z$ relatively prime, we let $\gamma_{m,n} \in \SL_2(\Z)$ be as in equation (\ref{eq_gammamn}).  By \cite{BS12}, the group $\SL_2(\Z)$ acts on $\E_{q,t}$, and we let 
$\iota^\E_{m,n}: \Lambda \to \E_{q,t}$
be the composition of the isomorphism $\Lambda \to \E_{q,t}^{hor}$ with the automorphism $\gamma_{m,n}: \E_{q,t} \to \E_{q,t}$. Since $\iota^\E_{m,n}$ is a map of algebras, it is uniquely determined by the following formula:
\begin{equation}
 \iota^\E_{m,n}(p_k) = v_{km,kn} \in \E_{q,t}
\end{equation}

We recall Macdonald's evaluation map \cite[eq.\ VI.6.17 and VI.8.8]{Mac95} in terms of $P_\lambda$:
\begin{equation}\label{eq_macev}
\ev^\E: \Lambda \to \C[q^\pmone,t^\pmone,u^\pmone],\quad\quad \ev^\E(P_\lambda) := \prod_{x \in \lambda} \frac{t^{l'(x)} - u q^{a'(x)}}{1 - q^{a(x)}t^{l(x)+1}}
\end{equation}

\begin{definition}\label{def_JE}
Given $m,n$ relatively prime and $\mm,\nn$ as above, define maps $\Lambda \to \Lambda$ via
\begin{eqnarray*}
\Gamma^\E_{m,n}(x) &:=& \iota_{m,n}^\E(x)\cdot 1\\
\Gamma^\E_{\mm,\nn} &:=& \Gamma^\E_{m_1,n_1}\circ \cdots \circ \Gamma^\E_{m_k,n_k}
\end{eqnarray*}
We then define the polynomial
\begin{equation*}
J^\E(\mm,\nn,\lambda; q,t,u) := \ev^\E(\Gamma^\E_{\mm,\nn}(P_\lambda))
\end{equation*}
\end{definition}

Before we prove the first equality of Theorem \ref{thm_iteratedcable} we prove a proposition relating the actions of $\E_{s^{-2},s^{-2}}$ and $H$ on $\Lambda$. We recall that Theorem \ref{thm_HisotoE} states that there is an isomorphism $H \to \E_{s^{-2},s^{-2}}$ uniquely determined by $P_\xx \mapsto s^{d(\xx)}v_\xx$. We twist this by a graded automorphism of $H$ to obtain the isomorphism
\begin{equation}\label{eq_newiso}
 \tilde \varphi: H \stackrel \sim \to \E_{s^{-2},s^{-2}},\quad \quad \tilde \varphi(P_{m,n}) = v^{n}s^{-m+d(m,n)}v_{m,n}
\end{equation}

\begin{proposition}\label{prop_compareactions}
 For $x \in \Lambda$ and $P \in H^>$, we have
 \[
  P\cdot x = \tilde \varphi(P)\cdot x
 \]
\end{proposition}
\begin{proof}
 We first compare the actions of $P_{m,n}$ and $\tilde \varphi(P_{m,n})$ on $\Lambda$ for certain $m,n$. First,
\begin{eqnarray*}
 P_{m,0}\cdot s_\lambda &=& p_m s_\lambda \\ 
 v^{-0}s^{-m+0+m}v_{m,0} &=& v_{m,0} s_\lambda = p_m s_\lambda
\end{eqnarray*}
Second, from Lemma \ref{lemma_actiononLambda}, for all $n$ we have
\[
P_{0,n}\cdot s_\lambda =  \left[\frac{v^{-n} - v^n}{s^n - s^{-n}} + v^n s^{-n} \sum_{i=1}^k (s^{-2n\lambda_i} - 1)s^{2ni}\right]s_\lambda 
\]
It is well-known that when $t=q$, the Macdonald polynomials specialize to the Schur functions. Then for $n > 0$, the $t=q=s^{-2}$ specialization of equation (\ref{eq_Eaction}) states
\begin{eqnarray*}
 v^n s^{0+n}v_{0,n}\cdot s_\lambda 
 &=& v^n s^{-n} \left[  \sum_i (s^{-2n \lambda_i} - 1)s^{2ni}\right]\\
v^{-n} s^{0+n}v_{0,-n}\cdot s_\lambda 
&=& v^{-n} s^n\left[ \sum_i (s^{2n \lambda_i} - 1)s^{-2ni}\right]
\end{eqnarray*}
We have therefore shown that $P_{m,0}\cdot s_\lambda = \tilde \varphi(P_{m,0})\cdot s_\lambda$ and that
\begin{equation}\label{eq_p0n}
 P_{0,n}\cdot s_\lambda = \left[ \frac{v^{-n} - v^n}{s^n - s^{-n}} + \tilde \varphi(P_{0,n})\right]\cdot s_\lambda
\end{equation}
Using the commutation relations in $H$, it is clear that any element in $H^>$ can be written as sums of products of commutators of $P_{m,0}$ and $P_{0,n}$. In the commutator the constant term on the right hand side of equation (\ref{eq_p0n}) drops out, which shows that
\begin{equation}
 P\cdot s_\lambda = \tilde \varphi(P)\cdot s_\lambda,\quad\quad\textrm{ for all } P \in H^>
\end{equation}
which completes the proof of the proposition.
\end{proof}

\begin{proof}(of equality (\ref{eq_equality1}) of Theorem \ref{thm_iteratedcable})
We first use the cabling formula in Proposition \ref{prop_cabling} to reduce the equality that we are supposed to prove to the following:
\begin{equation}\label{eq_wanttoshow}
\ev^\E(\Gamma^\E_{\mm,\nn}(P_\lambda))\big|_{q=s^{-2}, t=s^{-2}, u = v^2} = v^\bullet s^\bullet \ev^H(\Gamma^H_{\mm,\nn}(s_\lambda))
\end{equation}
where the powers depend on $\mm$, $\nn$, and $\lvert \lambda\rvert$. We recall that if $p_k \in \Lambda$ is a power sum, then $\iota^H_{m,n} = P_{km,kn} \in H^>$ and $\iota^\E_{m,n}(p_k) = v_{km,kn} \in \E^>_{s^{-2},s^{-2}}$. Then Proposition \ref{prop_compareactions} implies
\[
 \iota^H_{m,n}(p_k) = v^{kn}s^{-km}s^kv_{km,kn}(p_k) \in \End_R(\Lambda)
\]
where the equality is equality of operators in $\End_R(\Lambda)$. Now the assignments $p_k \mapsto v^{kn}s^{-km}s^k p_k$ induce a graded algebra isomorphism $\Lambda \to \Lambda$, which shows that for any homogeneous $x \in \Lambda$ (and in particular for $x = s_\lambda)$, we have the equality of operators
\[
 \iota^H_{m,n}(x) = v^{\lvert x \rvert n}s^{-\lvert x \rvert m}s^{\lvert x \rvert}v_{\lvert x \rvert m,\lvert x \rvert n}(x) \in \End_R(\Lambda)
\]
Since this equality is as operators on $\Lambda$, this shows that $\Gamma_{m,n}^H(x) = v^\bullet s^\bullet \Gamma^\E_{m,n}(x)$, where the powers depend on $m$, $n$, and $\lvert x \rvert$. Finally, the actions of $H$ and $\E$ on $\Lambda$ are graded: if $a \in H^>$, $b \in \E^>_{q,t}$, and $x \in \Lambda$ are homogeneous, then $a\cdot x$ and $b \cdot x$ are also homogeneous. (See Remark \ref{rmk_homogeneous}.) This implies that $\Gamma_{m,n}^H$ and $\Gamma_{m,n}^\E$ are homogeneous maps, which implies that for homogeneous $x \in \Lambda$,
\begin{equation}
 \Gamma_{\mm,\nn}^H(x) = v^\bullet s^\bullet \Gamma_{\mm,\nn}^\E(x)
\end{equation}
where the powers depend on $\mm$, $\nn$, and $\lvert x \rvert$.

Now to finish the proof of equation (\ref{eq_wanttoshow}), all that remains is to compare the evaluation maps $\ev^\E$ and $\ev^H$. We then equate parameters $t = q = s^{-2}$ and $u = v^2$ in the formula (\ref{eq_macev}) and compute
\begin{eqnarray*}
 \ev^\E(s_\lambda) &=& \prod_{x \in \lambda} \frac{s^{-2l'} - v^2 s^{-2a'} } {1 - s^{-2a}s^{-2l-2} }\\
 &=& v \prod_{x \in \lambda} \left(s^{1-l'-a'+a+l}\right) \frac{v^{-1}s^{a'-l'} - vs^{l'-a'}}{s^{-a-l-1}-s^{a+l+1} }\\
 &=& vs^{-\lvert \lambda\rvert} \prod_{x \in \lambda} \frac{v^{-1}s^{c(x)} - vs^{-c(x)}}{s^{hl(x)} - s^{-hl(x)}}\\
 &=& vs^{-\lvert \lambda \rvert} \ev^H(s_\lambda)
\end{eqnarray*}
where we have written $l = l(x)$, etc. (The third equality is a straightforward combinatorial identity.) Since the $s_\lambda$ are a homogeneous linear basis of $\Lambda$ and the maps $\Gamma_{\mm,\nn}^H$ and $\Gamma_{\mm,\nn}^\E$ are homogeneous, this shows equality (\ref{eq_wanttoshow}) and completes the proof.
\end{proof}

\subsection{Double affine Hecke algebras}\label{sec_dahapoly}
We now briefly recall the construction of \cite{CD14}. Since the main theorem of this section will follow from comparing to \cite{SV13}, we will only introduce the notation necessary to make this comparison. (As in the previous section, the $t$ of \cite{SV13} is our $t^{-1}$.)

The double affine Hecke algebra $\H_{q,t}^N$, of $\textrm{GL}_N$, abbreviated DAHA, is the algebra generated by elements $T_i^\pmone$ for $1 \leq i \leq N-1$, and $X_j^\pmone$, $Y_j^\pmone$ for $1 \leq j \leq N$, subject to some relations which we will not write down. This algebra is $\Z^2$-graded, with $deg(X_i) = (1,0)$, $deg(Y_i) = (0,1)$, and $deg(T_i) = 0$. There is an $\SL_2(\Z)$ action on $\H_{q,t}^N$ which permutes the graded components.

Let $S$ be the symmetrizing idempotent in the finite Hecke algebra (which is generated by the $T_i$'s), which is characterized by $T_jS = S T_j = t^{1/2}\e$ for all $j$. The spherical DAHA is the subalgebra $\SH^N_{q,t} := S \H^N_{q,t} S$ of $\H^N_{q,t}$, and it is also $\Z^2$-graded. The  $\SL_2(\Z)$ action on $\H^N_{q,t}$ preserves the subalgebra $\SH_{q,t}^N$.

Following \cite[Sec.\ 2.2]{SV11}, for $k > 0$ we define elements 
\begin{equation*}
 P^N_{0,k} = S \sum_i Y_i^k S
\end{equation*}
Elements $P^N_\xx$ for $\xx \in \Z^2$ are defined using the $\SL_2(\Z)$ action. We define $\SH_{q,t}^{N,>}$ to be the subalgebra of $\SH_{q,t}^N$ generated by $P_{m,n}^N$ with $m > 0$.

Cherednik defined an action of $\SH_{q,t}^N$ on $\Lambda^N$ using Demazure-Lusztig operators (see, e.g. \cite{Che95}). Instead of defining these operators, we recall the following theorem of Schiffmann and Vasserot. (This determines the action of $\SH_{q,t}^{N,>}$ on $\Lambda^N$ uniquely up to scalars, which is enough for our purposes.)

\begin{theorem}[\cite{SV13, SV11}]\label{thm_sv}
The assignment $v_\xx \mapsto P_\xx^N$ extends uniquely to a $\Z^2$-graded $\SL_2(\Z)$-equivariant surjective algebra homomorphism 
 \[
 \phi^N: \E_{q,t} \twoheadrightarrow \SH_{q,t}^{N}
 \] 
 Furthermore, under the projection $\pi_N: \Lambda \to \Lambda^N$, the actions of the subalgebras $\E_{q,t}^>$ and $\SH_{q,t}^{N,>}$ are related via the formula
 \[
  \left(q^{\bullet}t^{\bullet}u^{\bullet}\pi_N\circ v_{m,n}\right)\Big|_{u=t^N} = P_{m,n}^N \circ \pi_N
 \]
where $v_{m,n}$ and $P^N_{m,n}$ are viewed as endomorphisms of $\Lambda$ and $\Lambda^N$, respectively, and where the powers denoted `$\bullet$' depend on $m$ and $n$ but not on $N$ or on $\lvert \lambda \rvert$.
\end{theorem}
\begin{proof}
 The first statement is \cite[Thm.\ 3.1]{SV11}. The second statement follows from \cite[Lemma 1.3]{SV13} and the discussion directly preceding \cite[Prop.\ 1.4]{SV13}. We remark that the definition of the $P_{m,n}^N$ differs in \cite{SV13} and \cite{SV11} - we have chosen the latter because these make the map $\phi^N$ equivariant under the $\SL_2(\Z)$ action. The analogous surjection used in \cite{SV13} is a twist of $\phi^N$ by a graded automorphism, which accounts for the factor $q^\bullet$ in our statement. The factors $u^\bullet$ and $t^\bullet$ come from the two formulas below \cite[eq.\ (2.12)]{SV13}.
\end{proof}


There is a natural algebra map $\iota^N: \Lambda^N \to \SH_{q,t}^{N,>}$ which takes the power sum functions $p_k$ to the element $P^N_{k,0}$. Given $m,n \in \Z$ relatively prime with $m > 0$, we will write 
\[ \iota_{m,n}^N := \gamma_{m,n} \circ \iota^N
\]
 where $\gamma_{m,n} \in \SL_2(\Z)$ is as in equation (\ref{eq_gammamn}). We remark that the automorphism $\gamma_{m,n}$ of $\SH^N_{q,t}$ does not preserve the subalgebra $\SH_{q,t}^{N,>}$ - however, since $m > 0$, the image of the elements $P_{k,0}^N$ \emph{is} contained in the subalgebra $\SH_{q,t}^{N,>}$.

\begin{definition}[{\cite[eq.\ (2.13)]{CD14}}]\label{def_JN}
Given $m,n$ relatively prime with $m > 0$, sequences $\mm,\nn$ as above, and a partition $\lambda$ with at most $N$ parts, define maps $\Lambda^N \to \Lambda^N$ via
\begin{eqnarray*}
\Gamma^N_{m,n}(x) &:=& \iota^N_{m,n}(x)\cdot 1\\
\Gamma^N_{\mm,\nn} &:=& \Gamma^N_{m_1,n_1}\circ \cdots \circ \Gamma^N_{m_k,n_k}
\end{eqnarray*}
We then define the polynomial
\begin{equation*}
J^N(\mm,\nn,\lambda; q,t) := \ev^N(\Gamma^N_{\mm,\nn}(P_\lambda^N))
\end{equation*}
where $P^N_\lambda \in \Lambda^N$ is the Macdonald polynomial associated to the partition $\lambda$.
\end{definition}

\begin{remark}\label{rmk_rational}
The Macdonald polynomials are actually rational functions in $q$ and $t$, so the definition above actually produces a rational function. The definition in \cite{CD14} is a renormalization of the definition above (they divide by the evaluation of $P_\lambda^N$), and in their normalization the output of their formula is actually a polynomial, which is important for their purposes. The effect of this normalization is that the polynomials for the unknot are all 1. Since we work with the skein-theoretic normalization of the Homflypt polynomial, our choice of normalization is slightly more convenient for our purposes.
\end{remark}

\begin{proof}(of equation (\ref{eq_equality2}) of Theorem \ref{thm_iteratedcable})
We need to prove the equality
\[
 q^\bullet t^\bullet u^\bullet \ev^\E(\Gamma^\E_{\mm,\nn}(P_\lambda)) \Big|_{u=t^N} = \ev^N(\Gamma^N_{\mm,\nn}(P_\lambda^N))
\]
where the powers ``$\bullet$'' do not depend on $N$. We will do this by relating the various maps involving $\Lambda$ and $\Lambda^N$ with the projection $\pi_N:\Lambda \to \Lambda^N$ and $\phi^N: \E_{q,t} \to \SH_{q,t}^N$.

First, it is well known that $\pi_N(P_\lambda) = P^N_{\lambda}$. Since the surjection $\phi^N$ is $\SL_2(\Z)$-equivariant, we see that $\phi_N(\iota_{m,n}^\E(x)) = \iota_{m,n}^N(\pi_N(x))$ for $x \in \Lambda$. This combined with the assumption $m > 0$ and the second statement of Theorem \ref{thm_sv} shows the equality
\begin{equation}\label{eq_usemgeq0}
\left( q^\bullet t^{\bullet} u^\bullet \pi_N \circ \Gamma^\E_{m,n}\right)\Big|_{u=t^N} = \Gamma^N_{m,n}\circ \pi_N
\end{equation}
where the powers depend on $m,n$ but not on $N$. Finally, the equality 
\[
 \ev^\E\Big|_{u=t^N} = \ev^N\circ \pi_N
\]
completes the proof of the theorem.
\end{proof}

\bibliography{bibtex_arxiv}{}
\bibliographystyle{amsalpha}

\end{document}